\documentclass[10pt,a4paper]{amsart}
\usepackage[english]{babel}
\usepackage{amsmath,amsthm,mathrsfs,hyperref}
\usepackage{amssymb}
\usepackage{aliascnt}

\usepackage{comment}
\usepackage{todonotes}
\usepackage{xspace}
\usepackage{nicefrac}

\DeclareMathOperator{\dom}{dom}

\DeclareMathOperator{\crit}{crit}

\DeclareMathOperator{\cf}{cf}
\DeclareMathOperator{\stem}{stem}

\DeclareMathOperator{\Ult}{Ult}
\newcommand{\chang}{\twoheadrightarrow}

\newcommand{\Lowenheim}{L\"{o}wenheim\xspace}

\newcommand{\Erdos}{Erd\H{o}s\xspace}

\newcommand{\ZFC}{{\rm ZFC}\xspace}
\newcommand{\GCH}{{\rm GCH}\xspace}
\newcommand{\PCF}{{\rm PCF}\xspace}

\newtheorem{theorem}{Theorem}
\newaliascnt{example}{theorem}

\aliascntresetthe{example}
\newaliascnt{fact}{theorem}
\newtheorem{fact}[fact]{Fact}
\aliascntresetthe{fact}
\newaliascnt{corollary}{theorem}
\newtheorem{corollary}[corollary]{Corollary}
\aliascntresetthe{corollary}
\newaliascnt{lemma}{theorem}
\newtheorem{lemma}[lemma]{Lemma}
\aliascntresetthe{lemma}
\newaliascnt{claim}{theorem}
\newtheorem{claim}[lemma]{Claim}
\aliascntresetthe{claim}
\newtheorem{prop}[theorem]{Proposition}
\newtheorem{question}{Question}
\newtheorem{porism}[theorem]{Porism}

\newtheorem*{theorem*}{Theorem}
\newtheorem*{remark}{Remark}
\newtheorem*{example*}{Example}

\newtheorem*{definition}{Definition}

\DeclareMathOperator{\ran}{ran}
\DeclareMathOperator{\ot}{ot}

\DeclareMathOperator{\col}{Col}

\DeclareMathOperator{\ns}{NS}

\newcommand{\p}{\mathcal{P}}
\newcommand{\la}{\langle}
\newcommand{\ra}{\rangle}
\title{Global Chang's Conjecture and singular cardinals}
\author{Monroe Eskew}
\address[Monroe Eskew]{Universit\"{a}t Wien \\ Institut f\"{u}r Mathematik \\ Kurt G\"odel Research Center \\ Kolingasse 14-16 \\ 1090 Wien, Austria }
\email{monroe.eskew@univie.ac.at}

\author{Yair Hayut}
\address[Yair Hayut]{Department of Mathematics \\\
The Hebrew University of Jerusalem \\\
Jerusalem 9190401, Israel}
\email{yair.hayut@mail.huji.ac.il}

\makeatletter
\@namedef{subjclassname@2020}{%
  \textup{2020} Mathematics Subject Classification}
\makeatother


\begin{document}
\begin{abstract}
We investigate the possibilities of global versions of Chang's Conjecture that involve singular cardinals.  We show some $\ZFC$ limitations on such principles and prove relative to large cardinals that Chang's Conjecture can consistently hold between all pairs of limit cardinals below $\aleph_{\omega^\omega}$.
\end{abstract}
\maketitle

\section{Introduction}
The \Lowenheim-Skolem theorem asserts that for every pair of infinite cardinals $\kappa > \mu$ and every structure $\mathfrak{A}$ on $\kappa$ in a countable language, there is a substructure $\mathfrak B \subseteq \mathfrak A$ of size $\mu$.  ``Chang's Conjecture'' is a type of principle strengthening this theorem to assert similar relationships between sequences of cardinals.  For example $(\kappa_1,\kappa_0) \chang (\mu_1,\mu_0)$ says that for every structure $\mathfrak A$ on $\kappa_1$ in a countable language, there is a substructure $\mathfrak B$ of size $\mu_1$ such that $|\mathfrak B \cap \kappa_0| = \mu_0$.  The following basic observation puts some constraints on this type of principle:

\begin{prop}
\label{model arithmetic}
Suppose $\kappa,\lambda \leq \delta$ and $\kappa^\lambda \geq \delta$.  Then there is a structure $\mathfrak A$ on $\delta$ such that for every $\mathfrak B \prec \mathfrak A$, $$|\mathfrak B \cap \kappa|^{|\mathfrak B \cap \lambda|} \geq |\mathfrak B \cap \delta|.$$
\end{prop}

\begin{corollary}
\label{chang arithmetic}
If $(\kappa_1,\kappa_0) \chang (\mu_1,\mu_0)$, $\nu \leq \kappa_0$, and $\kappa_0^\nu \geq \kappa_1$, then $\mu_0^{\min(\mu_0,\nu)} \geq \mu_1$.
\end{corollary}

From this, we immediately see that under $\GCH$, $(\kappa^+,\kappa) \chang (\mu^+,\mu)$ can only occur when $\cf(\kappa) \geq \cf(\mu)$.  (The consistency of contrary cases is unknown.)  This inspires the following bold conjecture:
\begin{definition}[Global Chang's Conjecture]
We say that the Global Chang's Conjecture holds if for all infinite cardinals $\mu < \kappa$ with $\cf(\mu) \leq \cf(\kappa)$, $(\kappa^+,\kappa)\chang (\mu^+,\mu)$.
\end{definition}

In the paper \cite{EskewHayut2018}, we showed, assuming the consistency of a huge cardinal, that there is a model of $\ZFC+\GCH$ in which $(\kappa^+,\kappa)\chang (\mu^+,\mu)$ holds whenever $\kappa$ is regular and $\mu < \kappa$ is infinite.  Surprisingly, the full Global Chang's Conjecture is inconsistent (even without assuming $\GCH$), as we show in Theorem \ref{theorem: gcc is impossible}. Indeed, there is a tension between instances of Chang's Conjecture at successors of singular cardinals and at double successors of singulars.

Next, we investigate other forms of Global Chang's Conjecture: 
\begin{definition}[Singular Global Chang's Conjecture]
We say that the Singular Global Chang's Conjecture holds if for all infinite $\mu < \kappa$ of the same cofinality, $(\kappa^+,\kappa) \chang (\mu^+,\mu)$.
\end{definition}
Obtaining the Singular Global Chang's Conjecture seems to be hard. We present here a partial result, showing that there is a model in which the Singular Global Chang's Conjecture holds for cardinals below $\aleph_{\omega^\omega}$. 

The paper is organized as follows. In Section \ref{section:limitation} we discuss some relationships between Chang's Conjecture and $\PCF$-theoretic scales, and derive some $\ZFC$ limitations on the simultaneous occurrence of some instances of Chang's Conjecture.  In Section \ref{various}, we introduce the technology for obtaining $(\aleph_{\alpha+1},\aleph_\alpha)\chang (\aleph_{\beta+1},\aleph_\beta)$ for various choices of $\alpha$ and $\beta$ of countable cofinality.  In Section \ref{section: singular gcc} we construct a model in which $(\aleph_{\alpha+1},\aleph_\alpha)\chang (\aleph_{\beta+1},\aleph_\beta)$ holds for all limit ordinals $0 \leq \beta < \alpha < \omega^\omega$.  In Section \ref{section:thread}, we show the consistency of $(\aleph_{\alpha+1},\aleph_\alpha)\chang (\aleph_{\beta+1},\aleph_\beta)$ holding for a fixed $\beta$ but for $\alpha$ ranging over a longer interval of limit ordinals.  We conclude with some open questions.

\section{Limitations on Global Chang's Conjecture}\label{section:limitation}
A useful strengthening of Chang's Conjecture is the following, introduced by Shelah \cite{Shelah91}:
\begin{definition} We say $(\kappa_1,\kappa_0) \chang_\nu (\mu_1,\mu_0)$ if for all structures $\mathfrak A$ on $\kappa_1$ in a countable language, there is a substructure $\mathfrak B$ such that $| \mathfrak B | = \mu_1$, $| \mathfrak B \cap \kappa_0 | = \mu_0$, and $\nu \subseteq \mathfrak B$.
\end{definition}

Note that nothing more is asserted by adding the subscript $\nu$ when $\nu < \omega_1$.  These versions of Chang's Conjecture are robust under mild forcing:

\begin{lemma}
\label{preserve}
Suppose $(\kappa_1,\kappa_0) \chang_\nu (\mu_1,\mu_0)$ and $\mathbb P$ is a $\nu^+$-c.c.\ partial order.  Then $\Vdash_{\mathbb P} (\kappa_1,\kappa_0) \chang_\nu (\mu_1,\mu_0)$.
\end{lemma}

Of particular interest is the case $\nu = \mu_0$.   The following lemma gives a stepping-up of the Chang's Conjecture if the distance between the cardinals considered is not too great, or enough $\GCH$ holds relatively close to the upper end.  A proof is contained in \cite[Section 2.2.1]{Foreman2009}.

\begin{lemma}
\label{stepup}
Suppose $(\kappa_1,\kappa_0) \chang_{\nu} (\mu_1,\mu_0)$.
\begin{enumerate}
\item If $\kappa_0 = \mu_0^{+\nu}$, then $(\kappa_1,\kappa_0) \chang_{\mu_0} (\mu_1,\mu_0)$.
\item If $\lambda \leq \mu_0$ and there is $\kappa \leq \kappa_0$ such that $\kappa_0 = \kappa^{+\nu}$ and $\kappa^{\lambda} \leq \kappa_0$, then $(\kappa_1,\kappa_0) \chang_{\lambda} (\mu_1,\mu_0)$.
\end{enumerate}
\end{lemma}

When the hypotheses of the above lemma cannot be applied, some $\GCH$ at the lower end allows a similar conclusion in a special case.

\begin{lemma}
\label{bottomcard}
Suppose $\mu^{<\nu} = \mu$, and $(\kappa^+,\kappa) \chang (\mu^+,\mu)$.  Then $(\kappa^+,\kappa) \chang_\nu (\mu^+,\mu)$.
\end{lemma}

\begin{proof}
If $\kappa^\nu= \kappa$, then the conclusion follows from (2) of Lemma \ref{stepup}.  Otherwise, let $\mathfrak A$ be a structure on $\kappa^+$ which is isomorphic to a transitive elementary substructure of $(H_{\kappa^{++}}, \in, \lhd,\mu,\nu)$, where $\lhd$ is a well-order of $H_{\kappa^{++}}$.  It is easy to see that the conclusion of Proposition \ref{model arithmetic} applies to $\mathfrak A$ with respect to the cardinals $\kappa,\nu,\kappa^+$.  If $\mathfrak B \prec \mathfrak A$ witnesses Chang's Conjecture, then $|\mathfrak B \cap \kappa|^{|\mathfrak B \cap \nu|} = \mu^{|\mathfrak B \cap \nu|} \geq |\mathfrak B \cap \kappa^+| = \mu^+$.  Thus $|\mathfrak B \cap \nu| = \nu$.

Let $\delta \in \mathfrak B \cap \nu$.  Corollary \ref{chang arithmetic} implies that $\kappa^\delta = \kappa$.  Let $\la f_\alpha : \alpha < \kappa \ra \in \mathfrak B$ list all functions from $\delta$ to $\kappa$.  Let $\mathfrak B' = \text{Hull}^{\mathfrak A}(\mathfrak B \cup \delta).$  If $\beta \in \kappa \cap \mathfrak B'$, then there is function $f \in \, ^{\delta}\kappa \cap \mathfrak B$ and $\gamma < \delta$ such that $\beta = f(\gamma)$.  Thus $\mathfrak B' \cap \kappa = \{ f_\alpha(\gamma) : \alpha \in \mathfrak B \cap \kappa$ and $\gamma < \delta \}$, which has size $\mu$.  Now let $\mathfrak C = \text{Hull}^{\mathfrak A}(\mathfrak B \cup \nu).$  Since $\mathfrak B$ is cofinal in $\nu$, $\mathfrak C = \bigcup \{\text{Hull}^{\mathfrak A}(\mathfrak B \cup \delta) : \delta \in \mathfrak B \cap \nu \}$, so $|\mathfrak C \cap \kappa| = \mu$.
\end{proof}

Versions of Chang's Conjecture involving singular cardinals have a strong influence on the combinatorics in their neighborhood, even without cardinal arithmetic assumptions.  Recall that if $\kappa$ is singular, a \emph{scale for $\kappa$} is a collection of functions $\la f_\alpha : \alpha < \kappa^+ \ra$ contained in some product $\prod_{i < \cf(\kappa)} \kappa_i$, where $\la \kappa_i : i < \cf(\kappa) \ra$ is an increasing and cofinal sequence of regular cardinals below $\kappa$, such that the functions $f_\alpha$ are increasing and cofinal in the partial order of the product where we put $f < g$ when $| \{ i : f(i) \geq g(i) \} | < \cf(\kappa)$.  It is easy to construct scales under the assumption $2^\kappa = \kappa^+$, but Shelah proved in $\ZFC$ that scales exist for all singular cardinals (see \cite{AbrahamMagidorHandbook}).

A scale $\la f_\alpha : \alpha <\kappa^+ \ra$ is \emph{good at $\alpha$} when there is a sequence $\vec{g} = \la g_i : i < \cf(\alpha) \ra$ and $j_\star < \cf(\kappa)$, such that for all $j \geq j_\star$, $\langle g_i(j) \mid i < \cf(\alpha)\rangle$ is increasing and $\vec{g}$ and $\la f_\beta : \beta < \alpha \ra$ are interleaved (i.e., cofinal in each other).  A scale is \emph{bad at $\alpha$} when it is not good at $\alpha$.  A scale is \emph{better at $\alpha$} if there is a club $C \subseteq \alpha$ such that for all $\beta \in C$ there is $j < \cf(\kappa)$ such that $f_\gamma(i) < f_\beta(i)$ for $i \geq j$ and $\gamma\in C \cap \beta$.  Note that if $\cf(\alpha) > \cf(\kappa)$, then being better at $\alpha$ implies being good at $\alpha$.  A scale is simply called \emph{good} (or \emph{better}) if it is good (or better) at every $\alpha$ such that  $\cf(\alpha) > \cf(\kappa)$.  The key connection with Chang's Conjecture is the following (see \cite{ForemanMagidor97} or \cite{Shelah91}): 

\begin{lemma}
\label{cc no good scale}
If $\kappa$ is singular and $(\kappa^+,\kappa) \chang_{\cf(\kappa)} (\mu^+,\mu)$, then there is no good scale for $\kappa$.  Moreover, every scale $\la f_\alpha : \alpha < \kappa^+ \ra$ for $\kappa$ is bad at stationarily many $\alpha$ of cofinality $\mu^+$.
\end{lemma}

We now show that the full Global Chang's Conjecture is inconsistent with $\ZFC$.

\begin{lemma}
Suppose $\kappa$ is regular, $\mu < \kappa$ is singular, and $(\kappa^{+},\kappa) \chang (\mu^+,\mu)$.  Then $\mu$ carries a better scale. Moreover, if $\cf \mu = \omega$ then $\square_\mu^*$ holds.
\end{lemma}

\begin{proof}
Let us start with a general observation, following \cite[Theorem 2.15]{ForemanMagidor1995}. 
\begin{claim}
Let $\mu < \kappa = \cf(\kappa)$ be cardinals. Let $\theta$ be a regular cardinal above $\kappa^+$. If $H$ is the transitive collapse of some elementary substructure of $H_\theta$ of size $\kappa^+$ containing $\kappa^+$, and $M \prec H$ is such that $|M \cap \kappa^+| = \mu^+$ and $|M \cap \kappa| = \mu$, then $\cf(\sup(M \cap \kappa)) = \cf(\mu)$.  
\end{claim}
\begin{proof}
Fix in such an $H$ a sequence $\la x_\alpha : \alpha < \kappa^+ \ra$ of ``strongly almost disjoint" unbounded subsets of $\kappa$.  That is, for every $\alpha < \kappa^+$, there is a sequence $\la \gamma^\alpha_\beta : \beta < \alpha \ra \in H$ of ordinals below $\kappa$ such that $\la x_\beta \setminus \gamma^\alpha_\beta : \beta < \alpha \ra$ is pairwise disjoint. This principle, due to Shelah, is called $\mathrm{ADS}_\kappa$ and it holds for $\kappa$ regular (see \cite{CummingsForemanMagidor2001} and \cite{ShelahProper}).

Let $M \prec H$ be as above.  Let $f : \mu \to M \cap \kappa$ be a bijection.  If $\cf(\sup(M \cap \kappa)) \not= \cf(\mu)$, then for each $\alpha < M \cap \kappa^+$ there is $\delta_\alpha < \mu$ such that $f[\delta_\alpha] \cap x_\alpha$ is cofinal in $M \cap \kappa$.  Since $|M \cap \kappa^{+}| = \mu^+$, there is a set $Y \subseteq M \cap \kappa^+$ of size $\mu^+$ and a fixed $\delta <\mu$ such that $\delta_\alpha = \delta$ for all $\alpha \in Y$. Let $\zeta \in M \cap \kappa^{+}$ be large enough so that $|Y \cap \zeta| = \mu$. Note that $\langle \gamma^\zeta_\beta \mid \beta < \zeta\rangle \in M$ and thus for every $\beta \in M \cap \zeta$, $\gamma^\zeta_\beta \in M \cap \kappa$.

For $\beta \in Y \cap \zeta$, let $y_\beta = f[\delta] \cap x_\beta \setminus \gamma^\zeta_\beta$.  Then $\{ f^{-1}[y_\beta] : \beta \in Y \cap \alpha \}$ is a collection of $\mu$-many pairwise disjoint subsets of $\delta$, which is impossible.
\end{proof}

Let us return to the proof of the lemma.

By a theorem of Shelah \cite{Shelah91},
$\kappa$ carries a ``partial weak square'', a weak square sequence that misses only cofinality $\kappa$.  That is, there is a sequence $\la \mathcal C_\alpha : \alpha < \kappa^+ \ra$ such that whenever $\omega \leq \cf(\alpha) < \kappa$, then $\mathcal C_\alpha$ is a nonempty collection of size $\leq \kappa$ such that each $C \in \mathcal C_\alpha$ is a club subset of $\alpha$ of size $<\kappa$, and if $C \in \mathcal C_\alpha$ and $\beta \in \lim C$, then $C \cap \beta \in \mathcal C_\beta$.

Let $M \prec H$ be as above, with $\vec{\mathcal C} \in M$ a partial weak square at $\kappa$.  Let $\pi : M \to N$ be the transitive collapse.  Let $\vec{\mathcal D} = \pi(\vec{\mathcal C})$.  Since $\ot(M \cap \kappa^+) = \mu^+$ and $| M \cap \mathcal C_\alpha | \leq \mu$ for each $\alpha \in M \cap \kappa^+$, $\vec{\mathcal D}$ is a sequence $\la \mathcal D_\alpha : \alpha < \mu^+ \ra$, such that each $\mathcal D_\alpha$ has size $\leq \mu$, if $D \in \mathcal D_\alpha$ and $\beta \in \lim D$, then $D \cap \beta \in \mathcal D_\beta$, and $\mathcal D_\alpha$ is nonempty whenever $\alpha$ is a limit ordinal such that $\cf(\pi^{-1}(\alpha)) \not= \kappa$.  If $\alpha$ is such that $\cf(\pi^{-1}(\alpha)) = \kappa$, then there is an increasing cofinal map $f : \kappa \to \pi^{-1}(\alpha)$ in $M$, which implies that $\cf(\alpha) = \cf(\mu)$.  Therefore, $\mathcal D_\alpha$ is nonempty whenever $\cf(\alpha) \not= \cf(\mu)$.  Furthermore, if $D \in \mathcal D_\alpha$, then $\ot(D) < \pi(\kappa)$.

Next, we modify $\vec{\mathcal D}$ to a sequence $\vec{\mathcal E}$ with the same properties except that $|C| < \mu$ whenever $C \in \mathcal E_\alpha$ and $\alpha < \mu^+$.  It is easy to show by induction that for each $\eta < \mu^+$, there is a ``short square'' of length $\eta$---a coherent sequence of clubs $\la E_\alpha : \alpha < \eta \ra$ such that $|E_\alpha| < \mu$ for each $\alpha < \eta$.  Fix such a sequence $\la E_\alpha : \alpha < \pi(\kappa) \ra$.  For each $\alpha < \mu^+$, let $\mathcal E_\alpha = \{ \{ \beta \in D : \ot(D \cap \beta) \in E_{\ot(D)} \} : D \in \mathcal D_\alpha \}$.  Clearly each element of each $\mathcal E_\alpha$ has size $<\mu$.  If $C \in \mathcal E_\alpha$ and $\beta \in \lim C$, then there is $D \in \mathcal D_\alpha$ such that $\beta \in \lim D$ and $C = \{ \beta \in D : \ot(D \cap \beta) \in E_{\ot(D)} \}$.  Thus $D \cap \beta \in \mathcal D_\beta$ and $\ot(D \cap \beta) \in \lim E_{\ot(D)}$, so $C \cap \beta \in \mathcal E_\beta$.

Note that $\vec{\mathcal{E}}$ is a partial weak square, avoiding only ordinals of cofinality $\cf\mu$. Thus if $\cf \mu = \omega$, one can easily obtain a weak square sequence by completing the missing points in $\vec{\mathcal{E}}$.
 
Fix a scale for $\mu$, $\la f_\alpha : \alpha < \mu^+ \ra \subseteq \prod_{i < \cf(\mu)} \mu_i$.  Let us inductively construct a better scale $\la g_\alpha : \alpha < \mu^+ \ra$ as follows.  Let $g_0 = f_0$.  If $\mathcal E_\alpha$ is empty, let $g_\alpha = f_\gamma$, where $\gamma \geq \alpha$ and $f_\gamma$ eventually dominates $g_\beta$ for each $\beta < \alpha$.  If $\mathcal E_\alpha$ is nonempty, first, for all $C \in \mathcal E_\alpha$, define
 $$g_C(i) =  
 \begin{cases}
  \sup \{ g_\beta(i) + 1 : \beta \in C \} & \text{ if } \mu_i > |C|, \\
  0 & \text{ otherwise.}
 \end{cases}
$$
Then let $g_\alpha = f_\gamma$, where $\gamma \geq \alpha$ and $f_\gamma$ eventually dominates $g_\beta$ for each $\beta < \alpha$ and $g_C$ for each $C \in \mathcal E_\alpha$.

Clearly $\la g_\alpha : \alpha < \mu^+ \ra$ is a scale.  To check betterness, if $\cf(\alpha) > \cf(\mu)$, let $C \in \mathcal E_\alpha$.  If $\beta \in \lim C$, then 
$C \cap \beta \in \mathcal E_\beta$.  There is $i < \cf(\mu)$ such that $g_{C \cap \beta}(j) > g_\gamma(j)$ for $i < j < \cf(\mu)$ and $\gamma \in C \cap \beta$.  Thus if $C'$ is the set of limit points of some $C \in \mathcal E_\alpha$, then for all $\beta \in C'$ there is $i < \cf(\mu)$ such that $g_\beta(j) > g_\gamma(j)$ for $i < j < \cf(\mu)$ and $\gamma \in C' \cap \beta$.
\end{proof}

Combining the above with Lemmas \ref{bottomcard} and \ref{cc no good scale}, we have:

\begin{theorem}\label{theorem: gcc is impossible}
Suppose $\kappa$ is singular, $\lambda > \kappa$ is regular, $(\lambda^{+}, \lambda) \chang (\kappa^{+}, \kappa)$, and $\cf(\kappa) \leq \mu < \kappa$.  Then $(\kappa^{+}, \kappa) \not\chang_{\cf(\kappa)} (\mu^{+}, \mu)$.  Thus if $\mu^{<\cf(\kappa)} = \mu$, then $(\kappa^{+}, \kappa) \not\chang (\mu^{+}, \mu)$.
\end{theorem}

\begin{corollary}
$[\aleph_0,\aleph_\omega]$ is the maximal initial interval of cardinals on which the Global Chang's Conjecture can hold.
\end{corollary}
The negative direction follows from Theorem \ref{theorem: gcc is impossible} and the positive direction is proven in \cite[Section 5]{EskewHayut2018}.

It seems to be unknown whether $(\kappa^+,\kappa) \chang (\mu^+,\mu)$ is equivalent to $(\kappa^+,\kappa) \chang_\mu (\mu^+,\mu)$ for regular $\mu$.  However, further analysis of scales allows us to rule out some instances of Chang's Conjecture in $\ZFC$, and to show that these two notions are not in general equivalent for singular $\mu$, even under $\GCH$.  The authors are grateful to Chris Lambie-Hanson for showing us how to prove the following:

\begin{theorem}
Suppose $\kappa$ is a singular cardinal and $\vec f = \la f_\alpha : \alpha < \kappa^+ \ra$ is a scale for $\kappa$.  There is a club $C \subseteq \kappa^+$ such that for all regular cardinals $\mu,\nu$ such that $\cf(\kappa) < \mu < \mu^{+3} \leq \nu < \mu^{+\cf(\kappa)} \leq \kappa$, $\vec f$ is good at every $\alpha \in C$ of cofinality $\nu$.
\end{theorem}

\begin{proof}
Suppose $\cf(\kappa) < \mu < \mu^{+3} \leq \nu < \mu^{+\cf(\kappa)} \leq \kappa$.  By \cite[Theorem 2.21]{AbrahamMagidorHandbook}, there is a club $C_{\mu,\nu} \subseteq \kappa^+$ such that for every $\alpha \in C_{\mu,\nu}$ of cofinality $\nu$, $\la f_\beta : \beta < \alpha \ra$ has an exact upper bound $g$ such that $\cf(g(i)) \geq \mu$ for all $i$.  $g$ being an exact upper bound means that $g$ is an upper bound to $\la f_\beta : \beta < \alpha \ra$, and for every $h < g$, there is $\beta < \alpha$ such that $h < f_\beta$.  

The arguments for Lemmas 6--8 of \cite{MagidorShelah1994} show that $\cf(g(i)) = \nu$ on a cobounded set of $i<\cf(\kappa)$, which implies $\vec f$ is good at $\alpha$.  For the reader's convenience:  Let $\la \alpha_j : j < \nu \ra$ be cofinal in $\alpha$.  We cannot have that $\cf(g(i)) > \nu$ for all $i$ in an unbounded set $X \subseteq \cf(\kappa)$.  For then there would be an $i^* < \cf(\kappa)$, an unbounded $Y \subseteq \nu$, and an $h < g$ such that $f_{\alpha_j}(i) < h(i) <  g(i)$ for $i \in X \setminus i^*$ and $j \in Y$, contradicting that $g$ is an exact upper bound.  Thus there is some $\nu' \in [\mu,\nu]$ and an unbounded $X \subseteq \cf(\kappa)$ such that $\cf(g(i)) = \nu'$ for all $i \in X$.  Let $\la g_k : k < \nu' \ra$ be a pointwise increasing sequence such that $\sup_{k < \nu'} g_k(i) = g(i)$ for all $i \in X$.  Since $g$ is an exact upper bound, for each $k < \nu'$, there is $j < \nu$ such that $g_k \restriction X < f_{\alpha_j} \restriction X$.  Also, for each $j < \nu$, there is $i^* < \cf(\kappa)$ such that $f_{\alpha_j}(i) < g(i)$ for $i \in X \setminus i^*$, and thus some $k < \nu'$ such that $f_{\alpha_j} \restriction X < g_k \restriction X.$  This implies $\nu' = \nu$.

Finally, we can take the intersection of all the $C_{\nu,\mu}$ for regular $\nu,\mu < \kappa^+$ to get the desired club $C$.
\end{proof}

Therefore, if $\kappa$ is singular, $(\kappa^+,\kappa) \chang_{\cf(\kappa)} (\mu^+,\mu)$ fails whenever $\cf(\kappa)^{+3} \leq \mu < \cf(\kappa)^{+\cf(\kappa)}$. 
However, it is possible that the version of Chang's Conjecture holds when we drop the subscript ``$\cf(\kappa)$'' on the arrow:

\begin{prop}
Suppose there is a 3-huge cardinal.  Then there are singular cardinals $\lambda<\delta$ such that $\cf(\delta) < \lambda < \cf(\delta)^{+\cf(\delta)}$ and $(\delta^+,\delta) \chang (\lambda^+,\lambda)$.
\end{prop}

\begin{proof}
Let $j : V \to M$ have critical point $\kappa$, with $M^{j^3(\kappa)} \subseteq M$.  Let $\delta = j^2(\kappa)^{+j(\kappa)}$ and let $\lambda = j(\kappa)^{+\kappa}$.  Let $\mathfrak A$ be any structure on $\delta^+$.  In $M$, $j[\mathfrak A] \prec j(\mathfrak A)$, and we have that $|j[\mathfrak A]| = \delta^+$ and $|j[\mathfrak A] \cap j(\delta)| = \delta$.  Reflecting through $j$, we have that there is $\mathfrak B \prec \mathfrak A$ such that $| \mathfrak B | = j(\kappa)^{+\kappa+1} = \lambda^+$ and $| \mathfrak B \cap \delta | = j(\kappa)^{+\kappa} = \lambda$.
\end{proof}

\section{Chang's Conjecture between successors of various singulars}
\label{various}

Recall that a partial order is $(\kappa,\lambda)$-distributive if forcing with it adds no functions from $\kappa$ to $\lambda$.  The following lemma is a mild generalization of a lemma that was proved in \cite{EskewHayut2018}.
\begin{lemma}\label{lemma: reflecting cc}
Let $\gamma < \kappa$ be such that $\kappa^{+\gamma}$ is a strong limit cardinal and $\kappa$ is $\kappa^{+\gamma+1}$-supercompact, as witnessed by an embedding $j : V \to M$.  If $\mathcal U$ is the ultrafilter on $\kappa$ derived from $j$, then there is $A \in \mathcal U$ such that for every $\alpha < \beta$ in $A \cup \{\kappa \}$ and every iteration $\mathbb P *\dot{\mathbb Q}$ of size $< \beta^{+\gamma}$, such that $\mathbb P$ is $\alpha^{+\gamma+1}$-Knaster and $\Vdash_{\mathbb P} \dot{\mathbb Q}$ is $(\alpha^{+\gamma+1},\alpha^{+\gamma+1})$-distributive, 
$$\Vdash_{\mathbb{P} *\dot{\mathbb Q}} (\beta^{+\gamma + 1}, \beta^{+\gamma}) \chang_{\alpha^{+\gamma}} (\alpha^{+\gamma + 1}, \alpha^{+\gamma}).$$
\end{lemma}

\begin{proof}
We show that for a set $A \in \mathcal U$, for every $\alpha \in A$ and every iteration $\mathbb P * \dot{\mathbb Q}$ satisfying the hypothesis for $\beta = \kappa$ forces $(\kappa^{+\gamma + 1}, \kappa^{+\gamma}) \chang_{\alpha^{+\gamma}} (\alpha^{+\gamma + 1}, \alpha^{+\gamma})$.  Then standard reflection arguments yield the desired conclusion.  By Lemma \ref{stepup}, it suffices to prove that for all $\alpha \in A$, every such $\mathbb P * \dot{\mathbb Q}$ forces $(\kappa^{+\gamma+1},\kappa^{+\gamma}) \chang_\gamma (\alpha^{+\gamma + 1}, \alpha^{+\gamma})$, since by the assumptions that $\kappa^{+\gamma}$ is a strong limit and $|\mathbb P * \dot{\mathbb Q}|< \kappa^{+\gamma}$, it is forced that for some $\lambda \in [\kappa, \kappa^{+\gamma})$, $\lambda^\kappa < \kappa^{+\gamma}$, so we may increase the subscript to $\alpha^{+\gamma}$.  If the claim fails, then on a set $B \in \mathcal U$, for every $\alpha \in B$, there is an iteration $\mathbb P_\alpha * \dot{\mathbb Q}_\alpha$ and a name for a function $\dot f_\alpha : (\kappa^{+\gamma+1})^{<\omega} \to \kappa^{+\gamma}$ such that it is forced that for every $X \subseteq \kappa^{+\gamma+1}$ of size $\alpha^{+\gamma+1}$ with $\gamma \subseteq X$, the closure of $X$ under $\dot f_\alpha$ contains $\alpha^{+\gamma+1}$-many ordinals below $\kappa^{+\gamma}$.  We may assume that $\dot f_\alpha$ is forced to be closed under compositions.

In $M$, let $\mathbb P * \dot{\mathbb Q} = j(\la\mathbb P_\alpha * \dot{\mathbb Q}_\alpha : \alpha < \kappa \ra)(\kappa)$ and let $\dot f = j(\la \dot f_\alpha : \alpha < \kappa \ra)(\kappa)$.  Let $X = j[\kappa^{+\gamma+1}]$.  Note that $X$ is a subset of $j(\kappa^{+\gamma+1})$ containing $\gamma$ and of size $\kappa^{+\gamma+1}$.  By hypothesis, $\Vdash^M_{\mathbb P *\dot{\mathbb Q}} |\dot f[X^{<\omega}]| = \kappa^{+\gamma+1}$.  Since $j(\kappa^{+\gamma})$ is singular, it is forced that there is a sequence $\la \dot b_\alpha : \alpha < \kappa^{+\gamma+1} \ra \subseteq X$ such that $\la \dot f(\dot b_\alpha) : \alpha < \kappa^{\gamma+1} \ra$ is a strictly increasing sequence of ordinals below $j(\kappa^{+\xi})$, for some $\xi < \gamma$.  Let $\nu < \gamma$ and $(p_0,\dot q_0) \in \mathbb P * \dot{\mathbb Q}$ be such that $|\mathbb P * \dot{\mathbb Q}| < j(\kappa^{+\nu})$ and $(p_0,\dot q_0) \Vdash \dot f(\dot b_\alpha) < j(\kappa^{+\nu})$ for all $\alpha <  \kappa^{+\gamma+1}$.

Since $\dot{\mathbb Q}$ adds no subsets to $X$, there is a $\mathbb P$-name $\dot Y$ and a condition $(p_1,\dot q_1) \leq (p_0,\dot q_0)$ such that $(p_1,\dot q_1) \Vdash \la \dot b_\alpha : \alpha < \kappa^{+\gamma+1} \ra = \dot Y$.  Next, for each $\alpha < \kappa^{+\gamma+1}$, find $r_\alpha \leq p_1$ and $a_\alpha \in (\kappa^{+\gamma+1})^{<\omega}$ such that $r_\alpha \Vdash_{\mathbb P} j(\check a_\alpha) = \dot Y(\alpha)$.  Since $\mathbb P$ is $\kappa^{+\gamma+1}$-Knaster, there is $Z \subseteq \kappa^{+\gamma+1}$ of size $\kappa^{+\gamma+1}$ such that $r_\alpha$ and $r_\beta$ are compatible for $\alpha,\beta \in Z$.  Therefore, for $\alpha < \beta$ in $Z$, there is $r \in \mathbb P$ such that $(r,\dot q_1) \Vdash \dot f(j(\check a_\alpha)) < \dot f(j(\check a_\beta)) < j(\kappa^{+\nu})$.

Reflecting these statements to $V$, we have that for $\alpha < \beta$ in $Z$, there are $\gamma < \kappa$ and $(p,\dot q) \in \mathbb P_\gamma * \dot{\mathbb Q}_\gamma$ such that $|\mathbb P_\gamma * \dot{\mathbb Q}_\gamma| < \kappa^{+\nu}$ and  $(p,\dot q) \Vdash^V_{\mathbb P_\gamma * \dot{\mathbb Q}_\gamma} \dot f_\gamma(\check a_\alpha) < \dot f_\gamma(\check a_\beta) < \kappa^{+\nu}$.  This defines a coloring of $[\kappa^{+\gamma+1}]^2$ in $\kappa^{+\nu}$-many colors.  Since $\kappa^{+\gamma}$ is a strong limit, the \Erdos-Rado Theorem implies that there is a set $H \subseteq Z$ of size $\kappa^{+\nu+1}$ such that all pairs in $[H]^2$ get the same color.  Thus we have a fixed $\eta$ and a fixed $(p,\dot q) \in \mathbb P_\eta * \dot{\mathbb Q}_\eta$ such that $(p,\dot q) \Vdash \dot f_\eta(a_\alpha) < \dot f_\eta(a_\beta) < \kappa^{+\nu}$ for $\alpha < \beta$ in $H$.  This is a contradiction.
\end{proof}

\begin{corollary}\label{cor:cc-alpha+1-beta+1}
If there is a $(+\omega+1)$-supercompact cardinal, then there is a forcing extension in which $ (\aleph_{\alpha+1},\aleph_\alpha) \chang (\aleph_{\beta+1},\aleph_\beta)$ holds for all limit ordinals $0 \leq \beta < \alpha < \omega^2$.
\end{corollary}

\begin{proof}
Let $\kappa$ be $\kappa^{+\omega+1}$-supercompact, and let $A \subseteq \kappa$ be given by Lemma \ref{lemma: reflecting cc}.  Let $\la \alpha_i : i < \omega \ra$ enumerate the first $\omega$ elements of $A$.  Let 
$$\mathbb P = \col(\omega,\alpha_0^{+\omega}) \times \prod_{n<\omega} \col(\alpha_n^{+\omega+2},\alpha_{n+1}).$$
Clearly, $\mathbb P$ forces that $\alpha_n^{+\omega} = \aleph_{\omega \cdot n}$ for all $n$.  For a fixed $n$, we can factor $\mathbb P$ as $\mathbb Q_0 \times \col(\alpha_n^{+\omega+2},\alpha_{n+1}) \times \mathbb Q_1$.  By Lemma \ref{lemma: reflecting cc}, the product of the first two factors forces
  $(\alpha_{n+1}^{+\omega+1},\alpha_{n+1}^{+\omega}) \chang_{\alpha_{n}^{+\omega}} (\alpha_{n}^{+\omega+1},\alpha_{n}^{+\omega})$.  Since $\mathbb Q_1$ remains $\alpha_{n+1}^{+\omega+2}$-distributive after this, the instance of Chang's Conjecture is preserved.  
  Since Chang's Conjecture is transitive, i.e. $(\kappa_1,\kappa_0) \chang (\mu_1,\mu_0)$ and $(\mu_1,\mu_0) \chang (\nu_1,\nu_0)$ implies $(\kappa_1,\kappa_0) \chang (\nu_1,\nu_0)$, the conclusion follows.
\end{proof}

The limitation of our argument so far is that we only get Chang's Conjecture between successors of singulars for which there are tail-end sequences of cardinals below that are order-isomorphic.  We will overcome this with a forcing that collapses singular cardinals to onto others of different types while preserving their successors and the desired instances of Chang's Conjecture.

\begin{theorem}
\label{cc any countable cof}
Assume $\GCH$.  Suppose $\alpha<\beta$ are countable limit ordinals and $\kappa$ is $\kappa^{+\beta+1}$-supercompact.  Then there is a forcing extension in which $(\aleph_{\beta+1},\aleph_\beta) \chang (\aleph_{\alpha+1},\aleph_\alpha)$.
\end{theorem}

The proof breaks into cases depending on the ``tail types'' of $\alpha$ and $\beta$.  For ordinals $\alpha \geq \beta$, let $\alpha - \beta$ be the unique $\gamma$ such that $\alpha = \beta + \gamma$.  For an ordinal $\alpha$, let $\tau(\alpha)$ (the tail of $\alpha$) be $\min_{\beta<\alpha} (\alpha-\beta)$.  Let $\iota(\alpha)$ be the least $\beta$ such that $\alpha = \beta + \tau(\alpha)$.  An ordinal $\alpha$ is indecomposable iff $\alpha = \tau(\alpha)$, and all tails are indecomposable.  

\underline{Case 1}: $\tau(\alpha) = \tau(\beta) = \gamma$, or $\alpha = 0$.  Note that $\iota(\beta) \geq \alpha$, and let $\delta = \iota(\beta) - \alpha$.  Let $A \subseteq \kappa$ be given by Lemma \ref{lemma: reflecting cc} (with respect to $\gamma$).  Let $\zeta < \eta$ be in $A$, and force with $\col(\zeta^{+\gamma+\delta+2},\eta)$, so that the ordertype of the set of cardinals between $\zeta^{+\gamma}$ and $\eta^{+\gamma}$ becomes $\delta+\gamma$.  By Lemma \ref{lemma: reflecting cc}, we have $(\eta^{+\gamma+1},\eta^{+\gamma}) \chang_{\zeta^{+\gamma}} (\zeta^{+\gamma+1},\zeta^{+\gamma}).$  If $\alpha = 0$, force with $\col(\omega,\zeta^{+\gamma})$, and if $\alpha > 0$, force with $\col(\aleph_{\iota(\alpha)+1},\zeta)$.  In both cases, Chang's Conjecture is preserved, and we get $|\zeta^{+\gamma}| = \aleph_\alpha$ and $\eta^{+\gamma} = \aleph_{\alpha + \delta + \gamma} = \aleph_{\beta}$.

\hspace{1mm}

For the other cases, we will use a variation on the Gitik-Sharon forcing \cite{GitikSharon}, which singularlizes a large cardinal while collapsing a singular cardinal above it. 

The following definition is standard (see \cite{GitikHandbook}).
\begin{definition} A structure $\la \mathbb P,\leq,\leq^* \ra$ is a \emph{Prikry-type forcing} when $\leq$ and $\leq^*$ are partial orders of $\mathbb P$ (called \emph{extension} and \emph{direct extension} respectively), with $p \leq^* q \Rightarrow p \leq q$, and such that whenever $\sigma$ is a statement in the forcing language of  $\la \mathbb P,\leq \ra$ and $p \in \mathbb P$, then there is $q \leq^* p$ deciding $\sigma$.  Such a forcing is called \emph{weakly $\kappa$-closed} for a cardinal $\kappa$ if $\la \mathbb P, \leq^* \ra$ is $\kappa$-closed.
\end{definition}

It is easy to see that if $\mathbb P$ is of Prikry type and weakly $\kappa^+$-closed, then it is $(\kappa,\kappa)$-distributive.

Suppose $\gamma< \delta$ are limit ordinals of countable cofinality, and $\vec \gamma = \la \gamma_i : 1\leq i < \omega \ra$, $\vec \delta = \la \delta_i : 1\leq i < \omega \ra$ are sequences such that: 
\begin{enumerate}
\item $\vec \gamma$ is strictly increasing with $\sup_i \gamma_i = \gamma$.
\item $\vec \delta$ is nondecreasing with $\gamma \leq \delta_1$ and $\sum_i \delta_i = \delta$.
\end{enumerate}
Suppose $\kappa > \delta$ is $\kappa^{+\gamma_n}$-supercompact for each $n\geq 1$, and $\mu < \kappa$ is regular.  For $1\leq n<\omega$, let $U_n$ be a $\kappa$-complete normal measure on $\mathcal P_\kappa(\kappa^{+\gamma_n})$, and let $j_n : V \to M_n \cong \Ult(V,U_n)$ be the ultrapower embedding.  By the closure of the ultrapowers and $\GCH$, we may choose an $M_n$-generic $K_n \subseteq \col(\kappa^{+\delta_n+2},j_n(\kappa))^{M_n}$.  Let $\vec U = \la U_n : n < \omega \ra$ and $\vec K = \la K_n : n < \omega \ra$.

With these choices made, we may define the forcing $\mathbb P(\mu,\vec \gamma,\vec \delta,\vec U,\vec K)$, which will have the following properties: 
\begin{itemize}
\item The forcing is of Prikry type, weakly $\mu$-closed, and $\kappa^{+\gamma}$-centered (and thus has the $\kappa^{+\gamma+1}$-c.c.).
\item $\kappa$ is forced to become $\mu^{+\delta}$.
\item $(\kappa^{+\gamma})^V$ is collapsed to $\kappa$.
\end{itemize}

Conditions in $\mathbb P(\mu,\vec \gamma,\vec \delta,\vec U,\vec K)$ are sequences
$$\la f_{0},x_1,f_1,\dots,x_{n},f_{n},F_{n+1},F_{n+2},\dots\ra,$$
where:
\begin{enumerate}
 \item For $1\leq i \leq n$, $x_i \in \mathcal P_\kappa(\kappa^{+\gamma_i})$, and $\kappa_i := x_i \cap \kappa$ is inaccessible.
 \item For $1 \leq i < n$, $x_i \subseteq x_{i+1}$, and $\kappa_{i+1} > |x_i|$.
  \item $f_{0} \in \col(\mu,\kappa)$, and $\ran(f_0) \subseteq \kappa_1$ if $x_1$ is defined.
 \item For $1\leq i< n$, $f_i \in \col(\kappa_i^{+\delta_i+2},\kappa_{i+1})$.
 \item $f_{n} \in \col(\kappa_n^{+\delta_n+2},\kappa)$.
 \item For $i > n$, $\dom F_i \in U_i$.
  \item For $i > n$ and $x \in \dom F_i$, $x \supseteq x_n$ and $\kappa_x := x \cap \kappa$ is an inaccessible cardinal greater than $|x_{n}| + \sup(\ran f_{n})$.
 \item For $i > n$ and $x \in \dom F_i$, $F_i(x) \in \mathrm{Col}(\kappa_x^{+\delta_i+2},\kappa)$.
 \item For $i > n$, $[F_i]_{U_i} \in K_i$.
\end{enumerate}
Suppose $p = \la f_{0},\dots,x_{n},f_{n},F_{n+1},\dots\ra$ and $q = \la f'_{0},\dots,x'_{m},f'_{m},F'_{m+1},\dots\ra.$  We say $q \leq p$ when:
\begin{enumerate}
\item $m \geq n$.
\item $f'_i \supseteq f_i$ for $i \leq n$, and $x_i = x'_i$ for $1 \leq i \leq n$.
\item For $n < i \leq m$, $x'_i \in \dom F_i$ and $f'_i \supseteq F_i(x'_i)$.
\item For $i> m$, $\dom F'_i \subseteq \dom F_i$, and $F'_i(x) \supseteq F_i(x)$ for $x \in \dom F'_i$.
\end{enumerate}
For $p$ as above, let $\stem(p) = \la f_{0},\dots,x_{n},f_{n} \ra$, and say the \emph{length} of $p$ is $n$.  (The stem of a length-0 condition is of the form $\la f_{0} \ra$.)  

\begin{lemma}
\label{gs1 factor}
Suppose $\mu,\vec\gamma,\vec\delta,\vec U,\vec K$ are as above, and $p = \la f_0,x_1,\dots,x_n,f_n \ra ^\frown \vec F$ is a condition of length $n>0$.  Then $\mathbb P(\mu,\vec\gamma,\vec\delta,\vec U,\vec K) \restriction p$ is canonically isomorphic to
$$\col(\mu,\kappa_1) \restriction f_0 \times\dots\times \col(\kappa_{n-1}^{+\delta_{n-1}+2},\kappa_n) \restriction f_{n-1} \times \mathbb P(\kappa_n^{+\delta_n+2},\vec\gamma',\vec\delta',\vec U',\vec K') \restriction \la f_{n} \ra ^\frown \vec F',$$
where for each sequence $s \in \{\vec\gamma,\vec\delta,\vec U,\vec K,\vec F\}$, $s'$ is the sequence such that $s'(m) = s(n+m)$ for $m \geq 1$.
\end{lemma}

We say $q \leq^* p$ when $q \leq p$ and they have the same length.  If $q \leq p$ and $\stem(p)$ is an initial segment of $\stem(q)$, we say $q$ is an \emph{end-extension} of $p$, or $q \preceq p$.  Given a sequence $\vec F = \la F_i : 1 \leq i < \omega \ra$ such that $\la \emptyset \ra ^\frown \vec F$  is a condition of length 0, and another condition $p = \stem(p)^\frown \la H_i : n < i < \omega \ra$, define 
\begin{align*}
p \wedge \vec F :=  \stem(p)^\frown \la & \{ \la x,F_i(x) \cup H_i(x)\ra : x \in \dom F_i \cap \dom H_i \\
& \text{ and } F_i(x) \cup H_i(x) \text{ is a function}\} : n < i < \omega \ra.
\end{align*}
Note that $p \wedge \vec F$ is both $\preceq$ and  $\leq^* p$, but $p \wedge \vec F$ is not necessarily $\leq \la \emptyset \ra ^\frown \vec F$.  For a given stem $s$ and sequence $\vec F$ as above, we define $s \wedge \vec F = p \wedge \vec F$, where $p$ is the weakest condition with stem $s$.

It is easy to see that $\mathbb P(\mu,\vec \gamma,\vec \delta,\vec U,\vec K)$ is $\kappa^{+\gamma}$-centered, and a density argument shows that it forces all cardinals in $[\kappa,\kappa^{+\gamma}]$ to have countable cofinality. The fact that not more damage is done than intended is a consequence of the Prikry Property, which follows from a more basic combinatorial property.  If $\mathbb P$ is a partial order and $c : \mathbb P \to \{0,1,2\}$, we say $c$ is a \emph{decisive coloring} if whenever $c(p) >0$ and $q \leq p$, then $c(q) = c(p)$.

\begin{lemma}
\label{dc1}
Let $c$ be a decisive coloring of $\mathbb P(\mu,\vec \gamma,\vec \delta,\vec U,\vec K)$.
\begin{enumerate}
\item There is a sequence $\vec F$ such that for every condition $p$, every two $r,r' \preceq p \wedge \vec F$ of the same length have the same color.
\item For every condition $p$, there is $q \leq^* p$ such that every two  $r,r' \leq q$ of the same length have the same color.
\end{enumerate}
\end{lemma}

\begin{proof}
Let $\mathbb P = \mathbb P(\mu,\vec \gamma,\vec \delta,\vec U,\vec K).$  For (1), we prove the following claim by induction:  For all $n<\omega$ and all decisive colorings of the conditions of length $n$, there is $\vec F$ such that for all $m\leq n$ and every condition $p$ of length $m$, every two $r,r' \preceq p \wedge \vec F$ of length $n$ have the same color.  Suppose $n=0$ and $c$ is such a coloring.  For every $s \in \col(\mu,\kappa)$, choose if possible some $\vec F_s$ such that $c(\la s \ra ^\frown \vec F_s) >0$.  Using the closure of the higher collapses and diagonal intersections, we may select a single sequence $\vec F$ such that $\la s \ra \wedge \vec F \leq \la s \ra ^\frown \vec F_s$ for all $s$.  By decisiveness, $\vec F$ witnesses the claim for $n =0$.

Suppose the claim is true for $n-1$.  Let $c$ be any decisive coloring of the conditions of length $n$.  Using the closure of $\col(\kappa^{+\delta_n+2},j_{U_n}(\kappa))^{M_n}$, the genericity of $K_n$, and the decisiveness of $j_{U_n}(c)$, we can find a function $f^* \in K_n$ such that for every stem $s$ of length $n-1$, if there are some $g$ and $\vec F$ such that $g \supseteq f^*$ and $s ^\frown \la j_{U_n}[\kappa^{+\gamma_n}],g \ra ^\frown \vec F$ has color $>0$, then $s ^\frown \la j_{U_n}[\kappa^{+\gamma_n}],f^* \ra ^\frown \vec F$ already has this color.  If $F_n$ represents $f^*$, then for all stems $s$ of length $n-1$, there is $A_s \in U_n$ and a color $c_s <3$ such that for all $x \in A_s$, either there is $\vec F^{s,x} = \la F_{k}^{s,x} : n+1 \leq k < \omega \ra$ such that $s ^\frown \la x,F_n(x)  \ra ^\frown \vec F^{s,x}$ has color $c_s > 0$, or for all $x \in A_s$ and all $g \supseteq F_n(x)$, any $p$ of length $n$ with stem $s ^\frown  \la x,g \ra$ has color 0.  Let  $A$ be the diagonal intersection of the sets $A_s$.
Using the directed-closure of the filters $K_k$ and diagonal intersections, we may select a single sequence $\vec F$ that plays the role of $\vec F^{s,x}$ for all $s$ and $x$.  Putting $\vec F' = \la F_n \restriction A \ra ^\frown \vec F$, we have that for any condition $p$ of length $n-1$, all $q \preceq p \wedge \vec F'$ of length $n$ have the same color.  This defines a decisive coloring $c'$ of the conditions of length $n-1$ of the form $p \wedge \vec F'$, by coloring them whatever color an arbitrary length-$n$ end-extension receives.  By induction, there is $\vec F''$ such that for every $m \leq n-1$, for every condition $p$ of length $m$, every $q \preceq p \wedge \vec F''$ of length $n-1$ receives the same color under $c'$.  This means that every such $p \wedge \vec F''$ receives the same color under $c$ when end-extended to a condition of length $n$.

To finish the argument for (1), let $c$ be a decisive coloring of $\mathbb P$.  We have for each $n$ a sequence $\vec F_n$ such that the restriction of $c$ to conditions of length $n$ satisfies the inductive claim.  Using the countable closure of the filters $K_m$, we can find the desired $\vec F$ by taking a lower bound to all the conditions of the form $\la \emptyset \ra^\frown \vec F_n$.

For (2), let $\vec F$ be given by (1) and let $p \in \mathbb P$.  If there is $s \leq \stem(p)$ such that some end-extension of $s \wedge \vec F$ has color $>0$, then pick such an $s$ which achieves such a color $c^*$ by end-extending to length $n$, where $n$ is as small as possible.  Then every $r,r' \leq s ^\frown \vec F$ have color 0 if their length is $<n$, and color $c^*$ otherwise.
\end{proof}

\begin{corollary}
$\la \mathbb P(\mu,\vec \gamma,\vec \delta,\vec U,\vec K), \leq, \leq^* \ra$ is a Prikry-type forcing.  
\end{corollary}
\begin{proof}
If $\sigma$ is a sentence in the forcing language of $\mathbb P(\mu,\vec \gamma,\vec \delta,\vec U,\vec K)$, then we color a condition 0 if it does not decide $\sigma$, 1 if it forces $\sigma$, and 2 if it forces $\neg\sigma$.  This is decisive, so for every $p$, there is $q \leq^* p$ such that all extensions of $q$ of the same length have the same color.  If $q$ does not decide $\sigma$, then there are $r,r' \leq q$ of the same length forcing opposite decisions about $\sigma$, contradicting the property of $q$.
\end{proof}

\underline{Case 2 (of Theorem \ref{cc any countable cof})}: $\tau(\alpha) > \tau(\beta) = \gamma.$  Again, we have $\iota(\beta) \geq \alpha$, so let $\xi = \iota(\beta) - \alpha$.
Let $A \subseteq \kappa$ be given by Lemma \ref{lemma: reflecting cc} (with respect to $\gamma)$.  Find $\nu < \mu$ in $A$ such that $\nu$ is $\nu^{+\gamma+1}$-supercompact.  Let $G \subseteq \col(\nu^{+\gamma+\xi+2},\mu)$ be generic over $V$.  In $V[G]$, $(\mu^{+\gamma+1},\mu^{+\gamma}) \chang_{\nu^{+\gamma}}(\nu^{+\gamma+1},\nu^{+\gamma})$ holds, and $\nu$ is still  $\nu^{+\gamma+1}$-supercompact.  Let $\vec \gamma = \la \gamma_i : 1 \leq i < \omega \ra$ be an increasing sequence converging to $\gamma$.  Since $\tau(\alpha) > \gamma$, we may find a nondecreasing sequence $\vec \alpha = \la \alpha_i : 1 \leq i < \omega \ra$ such that $\gamma \leq \alpha_1$ and $\sum_i \alpha_i = \alpha$.

Since $\nu$ is $\nu^{+\gamma+1}$-supercompact, we can construct $\vec U$ and $\vec K$ as above according to the sequences $\vec\gamma,\vec\alpha$.
Let $H \subseteq \mathbb P(\omega,\vec\gamma,\vec\alpha,\vec U,\vec K)$ be generic over $V[G]$.  Since this forcing is $\nu^{+\gamma+1}$-c.c., Chang's Conjecture is preserved.  In the extension, $\nu = \aleph_\alpha$, $(\nu^{+\gamma+1})^{V[G]} = (\nu^+)^{V[G][H]}$, and $\mu^{+\gamma} = \aleph_{\alpha+\xi+\gamma} =\aleph_\beta$.

\hspace{1mm}

The third case requires a more detailed analysis of the Gitik-Sharon forcing.
Suppose $\mathbb P(\mu,\vec \gamma,\vec \delta,\vec U,\vec K)$ is built as above, around a sufficiently supercompact $\kappa$.  Associated to a generic filter $G$ are sequences $\la x_n : 1\leq n <\omega \ra$, and $\la C_n : n < \omega \ra$ determined by the stems of all conditions in $G$, where $C_{0}$ is generic for $\col(\mu,\kappa_1)$, and for $n\geq 1$, $C_n$ is generic for $\col(\kappa_n^{+\delta_n+2},\kappa_{n+1})$ and $x_n \in \mathcal P_\kappa(\kappa^{+\gamma_n})$.  From this sequence, we can recover $G$ by taking all conditions $\la f_{0},x_1,f_1,\dots,x_n,f_n,F_{n+1},\dots\ra$ such that:
\begin{enumerate}
\item $\la x_i : 1\leq i \leq n \ra$ is an initial segment of $\la x_i : 1\leq i < \omega \ra$.
\item For $i \leq n$, $f_i \in C_i$.
\item For $i > n$, $x_i \in \dom F_i$, and $F_i(x_i) \in C_i$.
\end{enumerate}
The collection of such conditions is a filter containing $G$, so it must equal $G$ by the maximality of generic filters.

\begin{lemma}
\label{autoproduct}
Let $V$ be a model of set theory, and let $\la \mathbb P_i,\kappa_i,G_i : i < n \ra$ be such that:
\begin{enumerate}
\item $\la \kappa_i : i < n \ra$ is an increasing sequence of regular cardinals in $V$.
\item For each $i$, $\mathbb P_i$ is a partial order in $V$ that is $(\kappa_i,\kappa_i)$-distributive and of size $\leq \kappa_{i+1}$.
\item For each $i$, $G_i$ is $\mathbb P_i$-generic over $V$.
\end{enumerate}
Then $\prod_{i<n}G_i$ is $\prod_{i<n}\mathbb P_i$-generic over $V$.
\end{lemma}
\begin{proof}
We show this by induction on $m \leq n$.  Suppose that $\prod_{i<m}G_i$ is $\prod_{i<m}\mathbb P_i$-generic over $V$.  Since $\mathbb P_m$ is $(\kappa_m,\kappa_m)$-distributive, forcing with it adds no antichains to $\prod_{i<m}\mathbb P_i$.  Thus $\prod_{i<m}G_i$ is $\prod_{i<m}\mathbb P_i$-generic over $V[G_m]$, and so $\prod_{i\leq m}G_i$ is $\prod_{i\leq m}\mathbb P_i$-generic over $V$.
\end{proof}

\begin{lemma}
\label{genchar}
$(\vec x,\vec C)$ generates a generic for $\mathbb P(\mu,\vec \gamma,\vec \delta,\vec U,\vec K)$ over $V$ iff the following hold:
\begin{enumerate}
\item For every sequence $\vec F = \la F_n : 1\leq n < \omega \ra$ such that $\la \emptyset \ra ^\frown \vec F$ is a condition of length 0, there is $m$ such that for all $n \geq m$, $x_n \in \dom F_n$ and $F_n(x_n) \in C_n$.
\item $C_{0}$ is generic for $\col(\mu,\kappa_1)$, and $C_n$ is generic for $\col(\kappa_n^{+\delta_n+2},\kappa_{n+1})$ for all $n>0$.
\end{enumerate}
\end{lemma}

\begin{proof}
The forward direction is clear.  For the reverse direction, let $D \in V$ be a dense open subset of $\mathbb P = \mathbb P(\mu,\vec \gamma,\vec \delta,\vec U,\vec K)$, and let $G$ be the filter generated by $(\vec x,\vec C)$.  Let $c : \mathbb P \to 2$ be defined by $c(p) = 0$ if $p \notin D$ and $c(p) = 1$ otherwise.  This is decisive, so let $\vec F$ be given by Lemma \ref{dc1}.  Let $m$ be given by (1).

Consider the condition $p = \la \emptyset, x_1,\emptyset,\dots,x_{m-1},\emptyset \ra ^\frown \la F_i : m \leq i < \omega \ra$.  Let $D' = \{ q \in D : q \leq p \}$.  $D'$ projects to a dense subset of $\col(\mu,\kappa_1) \times  \col(\kappa_1^{+\delta_1+2},\kappa_{2}) \times \dots \times \col(\kappa_{m-1}^{+\delta_{m-1}+2},\kappa_{m})$. 
By (2) and Lemma \ref{autoproduct}, there is a sequence $\la f_i : i < m \ra$ that is in the projection of $D'$ intersected with $C_{0} \times \dots \times C_{m-1}$.  Thus there is some condition of the form 
$$\la f_{0},x_1,f_1,\dots,x_{m-1},f_{m-1},y_m,f_m,\dots,y_n,f_n,F'_{n+1},\dots\ra$$
that is in $D'$.  But by the homogeneity property of $\vec F$, we also have that
$$\la x_0,f_0,\dots,x_{m-1},f_{m-1},x_m,F_m(x_m),\dots,x_n,F_n(x_n),F_{n+1},\dots\ra \in D.$$
Therefore, $D \cap G \not=\emptyset$.
\end{proof}

\underline{Case 3 (of Theorem \ref{cc any countable cof})}: $0<\tau(\alpha) = \gamma < \tau(\beta).$  Let $\delta = \beta - \iota(\alpha)$.  We can find a nondecreasing sequence $\vec \delta = \la \delta_i : 1 \leq i < \omega \ra$ such that $\delta_1 \geq \gamma$ and $\sum_i \delta_i = \delta$.   Let $\vec \gamma = \la \gamma_i : 1 \leq i < \omega \ra$ be an increasing sequence converging to $\gamma$.  Let $j$ be an embedding witnessing that $\kappa$ is $\kappa^{+\gamma+1}$-supercompact, and let $A \subseteq \kappa$ be given by Lemma \ref{lemma: reflecting cc} (with respect to $\gamma)$.  For each $n \geq 1$, let $U_n$ be a $\kappa$-complete normal measure on $\p_\kappa(\kappa^{+\gamma_n})$ derived from $j$, so that $A$ is in the projection of each $U_n$ to $\kappa$. Let $\mu = \aleph_{\iota(\alpha)+1}$, and let us force with $\mathbb P = \mathbb P(\mu,\vec\gamma,\vec\delta,\vec U,\vec K)$ for where $\vec K$ is a sequence of filters as in the construction.

Let $p_0$ be a condition of length 0 forcing every Prikry point to be in $A$.  Let $p_1 \leq p_0$ be a condition of length 1 deciding the statement $\sigma :=$ ``$(\kappa^+,\kappa) \chang (\mu^{+\gamma+1},\mu^{+\gamma})$.''  We claim $p_1 \Vdash \sigma$.

Let us define an iteration of ultrapowers.  Let $N_1 = V$.  Given a commuting system of elementary embeddings $j_{m,m'} : N_m \to N_{m'}$ for $1\leq m \leq m' \leq n$, let $j_{n,n+1} : N_n \to \Ult(N_n,j_{1,n}(U_{n+1})) = N_{n+1}$ be the ultrapower embedding, and let $j_{m,n+1} = j_{n,n+1} \circ j_{m,n}$ for $1\leq m < n$.  For $1\leq n < \omega$, let $j_{n,\omega} : N_n \to N_\omega$ be the direct limit embedding.  
$N_\omega$ is well-founded, and thus can be identified with a transitive class, because of the following generalization of a well-known theorem of Gaifman (see \cite{Steel2016}). 

\begin{fact}\label{fact: gaifman}
If $\mathcal E$ is a set of countably complete ultrafilters, and $j_{\alpha,\beta} : N_\alpha \to N_\beta$, $\alpha < \beta \leq \theta$, is a system of elementary embeddings defined by taking at each $\alpha<\theta$ the ultrapower map $j_{\alpha,\alpha+1} : N_\alpha \to \Ult(N_\alpha,U) = N_{\alpha+1}$ for some $U \in j_{0,\alpha}(\mathcal E)$, and taking direct limits at limit stages, then each $N_\alpha$ is well-founded.
\end{fact}

Let $\stem(p_1) = \la f_{0},x_1,f_1 \ra$, and let $C_{0} \times C_1 \subseteq \col(\mu,\kappa_1) \times \col(\kappa_1^{+\delta_1+2},\kappa)$ be a filter that contains $\la f_{0},f_1\ra$ and is generic over $V$.  For $n>1$, let $y_n = j_{n-1,n}[ j_{1,n-1}(\kappa^{+\gamma_n})]$, and let $x_n = j_{n,\omega}(y_n)$, and let $C_n = j_{1,n-1}(K_n)$.

\begin{claim}\label{claim: generic over ultrapower, simple}
$\la x_n : 1 \leq n < \omega \ra$ and $\la C_n : n < \omega \ra$ together generate a generic filter for $j_{1,\omega}(\mathbb P)$ over $N_\omega$.
\end{claim}

\begin{proof}
We need to verify the two conditions of Lemma \ref{genchar}.  For (1), suppose $\vec F = \la F_n : 1\leq n < \omega \ra$ is such that $\la \emptyset \ra ^\frown \vec F \in j_{1,\omega}(\mathbb P)$ is a condition of length 0.  Let $m < \omega$ be such that $\vec F =  j_{m,\omega}(\vec F')$ for some $\vec F'$.  For $n \geq m$, $\dom j_{m,n}(F'_{n+1}) \in j_{1,n}(U_{n+1})$, and $N_n \models [j_{m,n}(F'_{n+1})]_{j_{1,n}(U_{n+1})} \in j_{1,n}(K_{n+1})$.  Thus for $n \geq m$, $y_{n+1} \in \dom j_{m,n+1}(F'_{n+1})$, and $f_{n+1} := j_{m,n+1}(F'_{n+1})(y_{n+1}) \in C_{n+1}$.  Note that $f_{n+1}$ is an object of rank $< j_{1,n+1}(\kappa) = \crit(j_{n+1,\omega})$.  Thus for $n > m$, $x_n \in \dom F_n$ and $f_n = F_n(x_n) \in C_n$.

To verify (2), note that for each $n > 1$, $N_{n-1} \models j_{1,n-1}(K_n)$ is generic for $\col(j_{1,n-1}(\kappa^{+\delta_n+2}),j_{1,n}(\kappa))$ over $N_n$.  It is also generic over the submodel $N_\omega$.  Note also for each $n > 1$, $\kappa_n := x_n \cap j_{1,\omega}(\kappa) = j_{1,n-1}(\kappa)$.
\end{proof}

Let $G$ be the generated filter for $j_{1,\omega}(\mathbb P)$.  Note that $j_{1,\omega}(p_1) \in G$.  We claim that $N_\omega[G]$ is closed under $\kappa$-sequences from $V[C_0 \times C_1]$.  Since $C_0 \times C_1$ is generic for a forcing of size $\kappa$, it suffices to show that $N_\omega[\la x_n : 2 \leq n < \omega \ra]$ is closed under $\kappa$-sequences from $V$, an idea due to Bukovsky \cite{Bukovsky1977} and independently to Dehornoy \cite{Dehornoy1978}.  This follows from the fact that every element of $N_\omega$ is of the form $j_{1,\omega}(f)(x_{2},\dots,x_n)$ for some function $f \in V$ and some $n<\omega$.  Let $\la f_\alpha : \alpha < \kappa \ra$ be a sequence of functions in $V$, such that for each $\alpha$, there is $n_\alpha$ such that $\dom f_\alpha = \p_\kappa(\kappa^{+\gamma_{2}}) \times \dots \times  \p_\kappa(\kappa^{+\gamma_{n_\alpha}})$.  Then $\la j_{1,\omega}(f_\alpha)(x_{2},\dots,x_{n_\alpha}) : \alpha < \kappa \ra$ can be computed from $j_{1,\omega}(\la f_\alpha : \alpha < \kappa \ra)$ and $\la x_n : 2 \leq n < \omega \ra$.

For all $\alpha < j_{1,\omega}(\kappa)$, there are $n<\omega$ and $\beta <j_{1,n}(\kappa)$ such that $\alpha = j_{n,\omega}(\beta)$, and $\alpha = \beta$ since $\crit(j_{n,\omega}) = j_{1,n}(\kappa)$.  By $\GCH$ and the nature of the measures, for $2\leq n<\omega$, $\kappa^{+\gamma_{n}} <  j_{1,n}(\kappa) < \kappa^{+\gamma}$.  Therefore, $j_{1,\omega}(\kappa) = \kappa^{+\gamma}$.  Furthermore, an easy counting argument shows that $j_{1,\omega}(\kappa^{+\gamma+1}) = \kappa^{+\gamma+1}$.

By Lemma \ref{lemma: reflecting cc}, $V[C_0 \times C_1] \models (\kappa^{+\gamma+1},\kappa^{+\gamma}) \chang (\mu^{+\gamma+1},\mu^{+\gamma})$.  Let $\mathfrak{A} \in N_\omega[G]$ be an algebra on $\kappa^{+\gamma+1} = (j_{1,\omega}(\kappa)^+)^{N_\omega[G]}$.  In $V[C_0 \times C_1]$, there is $\mathfrak{B} \prec \mathfrak{A}$ of size $\mu^{+\gamma+1}$ such that $| \mathfrak{B} \cap \kappa^{+\gamma} | = \mu^{+\gamma}$.  By the closure of $N_\omega[G]$, $\mathfrak{B} \in N_\omega[G]$.  This shows that $N_\omega[G]$ satisfies the desired instance of Chang's Conjecture, and thus by elementarity that $p_1$ forces $(\kappa^+,\kappa) \chang (\mu^{+\gamma+1},\mu^{+\gamma})$.  This completes the proof of Theorem \ref{cc any countable cof}.

\begin{corollary}
Suppose $\mathbb P = \mathbb P(\mu,\vec\gamma,\vec\delta,\vec U,\vec K)$ is as above.  Then there is a condition $p \in \mathbb P$ of length 0 that forces 
$$(\mu^{+\delta+1},\mu^{+\delta}) \chang (\mu^{+\sum_1^n \delta_i+\gamma+1},\mu^{+\sum_1^n \delta_i+\gamma}) \chang$$
$$(\mu^{+\sum_1^m \delta_i+\gamma+1},\mu^{+\sum_1^m \delta_i+\gamma})\chang (\mu^{+\gamma+1},\mu^{+\gamma})$$
 for $1\leq m < n<\omega$.
\end{corollary}
\begin{proof}
Note that it is forced that $\mu^{+\gamma}=\kappa_1^{+\gamma}$, and for each $n\geq 1$, $\kappa_n^{+\delta_n+\gamma} = \kappa_{n+1}^{+\gamma} = \mu^{+\sum_1^n \delta_i+\gamma}$.  Let $p$ be a condition of length 0 that forces all Prikry points to be in the set $A$ given by Lemma \ref{lemma: reflecting cc}.     Fix $1 \leq m < n < \omega$, and let $q \leq p$ be a condition of length $n$.  By Lemma \ref{gs1 factor}, $\mathbb P \restriction q$ is isomorphic to a restriction of
$$\col(\mu,\kappa_1) \times \dots \times \col(\kappa_{n-1}^{+\delta_{n-1}+2},\kappa_{n}) \times \mathbb P(\kappa_{n}^{+\delta_{n}+2},\vec\gamma',\vec\delta',\vec U',\vec K'),$$
where $s'$ denotes the shift of a sequence $s$ by $n$. By Lemma \ref{lemma: reflecting cc}, this product forces 
$(\kappa_{n}^{+\gamma+1},\kappa_{n}^{+\gamma}) \chang (\kappa_{m}^{+\gamma+1},\kappa_{m}^{+\gamma}) \chang (\mu^{+\gamma+1},\mu^{+\gamma}).$
The last two terms of the product are isomorphic to a restriction of
$\mathbb P(\kappa_{n-1}^{+\delta_{n-1}+2},\vec\gamma'',\vec\delta'',\vec U'',\vec K'')$ to a condition of length 1, where $s''$ denotes the shift of the original sequence $s$ by $n-1$.
By the argument for Case 3 of Theorem \ref{cc any countable cof}, this forces $(\kappa_n^{+\sum^\infty_{n}\delta_i+1},\kappa_n^{+\sum^\infty_{n}\delta_i}) \chang (\kappa_n^{+\gamma+1},\kappa_n^{+\gamma}).$
\end{proof}

Our methods are not limited to getting $(\aleph_{\beta+1},\aleph_\beta) \chang (\aleph_{\alpha+1},\aleph_\alpha)$ where $\alpha$ and $\beta$ are countable.  For example, if we opt not to interleave collapses in the Gitik-Sharon forcing, we obtain:

\begin{porism}
Let $\alpha \geq \omega$ be a countable limit ordinal, 
and let $\kappa$ be a $\kappa^{+\alpha+1}$-supercompact cardinal.  Then there is a generic extension in which $(\lambda^+,\lambda) \chang (\aleph_{\alpha+1},\aleph_\alpha)$, and another in which $(\lambda^{+\alpha+1},\lambda^{+\alpha}) \chang (\lambda^+,\lambda)$, where in both cases $\cf(\lambda) = \omega$ and $\aleph_\lambda = \lambda$.
\end{porism}

\section{Singular Global Chang's Conjecture below \texorpdfstring{$\aleph_{\omega^\omega}$}{alephomegaomega}}\label{section: singular gcc}
In this section we will prove the following theorem:
\begin{theorem}
If there is a model of \ZFC with a cardinal $\delta$ which is $\delta^{+\omega+1}$-supercompact and Woodin for supercompactness, then there is a model in which $(\aleph_{\alpha+1},\aleph_\alpha) \chang (\aleph_{\beta+1},\aleph_\beta)$ holds for all limit $\beta < \alpha < \omega^\omega$ (including $\beta=0$).
\end{theorem}
This theorem is an attempt to strengthen Corollary \ref{cor:cc-alpha+1-beta+1}, into a global result. Unfortunately, we do not know how to obtain the desired global result, or even the more natural one in which Chang's Conjecture holds between $(\aleph_{\alpha + 1}, \aleph_\alpha)$ and $(\aleph_{\beta + 1}, \aleph_{\beta})$ for all $\beta < \alpha$ countable limit ordinals. We believe that this is a limitation of our method and not an actual $\ZFC$-barrier. 

Before diving into the technical details, let us sketch the main ideas behind the forcing construction: After a suitable preparation, we obtain a model in which many instances of Chang's Conjecture occur between pairs of cardinals of the form $\kappa^{+\omega}$ and its successor and $\mu^{+\omega}$ and its successor. 
In this model we also have many supercompact cardinals, and this is the reason that we start with a stronger large cardinal hypothesis.

In order to obtain more instances of Chang's Conjecture, we need to apply the ``tail changing'' forcing, which is a Prikry-type forcing resembling the Gitik-Sharon forcing \cite{GitikSharon}. Since we would like to do that simultaneously for more than a single pair of cardinals, we define a Magidor- or Radin-like variant of the Gitik-Sharon forcing. Unfortunately, the diagonal nature of the forcing does not allow us to use a Mitchell-increasing sequence of measures, and we are forced to let the domain of measures increase (a similar issue was encountered in \cite{BenNeriaLambieHansonUnger}). This limits the result of the theorem.

\begin{definition}
A cardinal $\delta$ is called \emph{Woodin for supercompactness} when for every $A \subseteq \delta$, there is $\kappa < \delta$ such that for all $\lambda \in [\kappa,\delta)$, there is a normal $\kappa$-complete ultrafilter $U$ on $\p_\kappa(\lambda)$ such that $j_U(A) \cap \lambda = A \cap \lambda$.
\end{definition}

Like Woodin cardinals, Woodin for supercompactness cardinals need not be even weakly compact, but they have higher consistency strength than supercompact cardinals.  Every almost-huge cardinal is Woodin for supercompactness. Woodin for supercompact cardinals are the same as Vop\v{e}nka cardinals (see \cite{Perlmutter2015}).

\begin{lemma}
\label{prepmodel}
 Suppose $\GCH$ and $\delta$ is $\delta^{+\omega+1}$-supercompact and Woodin for supercompactness.  Then there is a model of $\ZFC$ in which $\GCH$ holds, there is a supercompact cardinal, and 
 for all $\alpha< \beta$, 
 \[(\beta^{+\omega+1},\beta^{+\omega}) \chang (\alpha^{+\omega+1},\alpha^{+\omega}).\] 
 Furthermore, any such instance of Chang's Conjecture is preserved by forcing over this model with a $(\alpha^{+\omega+1},\alpha^{+\omega+1})$-distributive forcing of size $<\beta^{+\omega}$.
 \end{lemma}
 
 \begin{proof}
 Let $A \subseteq \delta$ be given by Lemma \ref{lemma: reflecting cc}.  Let $\la \alpha_i : i < \delta \ra$ enumerate the closure of $A$.  Force with the following Easton support iteration $\langle \mathbb P_i, \dot{\mathbb Q}_j : i \leq \delta, j < \delta \rangle$: 
\begin{enumerate}
\item $\mathbb Q_0 = \col(\omega,\alpha_0^{+\omega}) \ast \dot{\col}(\alpha_0^{+\omega + 2}, \alpha_1)$. \item If $i > 0$ and $\alpha_i \in A$, $\Vdash_i \dot{\mathbb Q}_i = \dot{\mathrm{Col}}(\alpha_i^{+\omega+2},\alpha_{i+1})$.
 \item If $i > 0$ and $\alpha_i \notin A$, $\Vdash_i \dot{\mathbb Q}_i = \dot{\mathrm{Col}}(\alpha_i^+,\alpha_{i+1})$.
\end{enumerate}
It is easy to see that this iteration forces that for all infinite $\alpha<\delta$,
\[\left(\alpha^{+\omega}\right)^{V^{\mathbb{P}_\delta}} = \left(\beta^{+\omega}\right)^V,\] for some $\beta \in A$.  By standard arguments, $\delta$ remains inaccessible in $V^{\mathbb{P}_\delta}$.

Suppose that in $V^{\mathbb P_\delta}$, $\alpha < \alpha^{+\omega} < \beta<\delta$, and let $i<j$ be such that \[\left(\alpha^{+\omega}\right)^{V^{\mathbb{P}_\delta}} = \left(\alpha_i^{+\omega}\right)^V\text{ and }\left(\beta^{+\omega}\right)^{V^{\mathbb{P}_\delta}} = \left(\alpha_j^{+\omega}\right)^V.\]  Then $\mathbb P_\delta$ factors as $\mathbb P_i *  \mathbb P_j/\mathbb P_i * \mathbb P_\delta / \mathbb P_j$, where $|\mathbb P_i| \leq \alpha_i^{+\omega}$, $\mathbb P_j/\mathbb P_i$ is forced to be $\alpha_i^{+\omega+2}$-closed and of size $\leq \alpha_j$, and $\mathbb P_\delta / \mathbb P_j$ is forced to be $\alpha_j^{+\omega+2}$-closed.

Suppose $\mathbb Q$ is an $(\alpha_i^{\omega+1},\alpha_i^{\omega+1})$-distributive forcing of size $<\alpha_j^{+\omega}$ in $V^{\mathbb P_\delta}$.  Then $\mathbb Q \in V^{\mathbb P_j}$.  Since $\mathbb P_i$ forces that $\mathbb P_j/ \mathbb P_i * \mathbb Q$ is $(\alpha_i^{+\omega+1},\alpha_i^{+\omega+1})$-distributive, Lemma \ref{lemma: reflecting cc} implies that $\mathbb P_j * \mathbb Q$ forces $(\alpha_j^{+\omega+1},\alpha_j^{+\omega}) \chang (\alpha_i^{+\omega+1},\alpha_i^{+\omega})$.  This is preserved by $\mathbb P_\delta / \mathbb P_j$, which remains $(\alpha_j^{+\omega+1},\infty)$-distributive after forcing with $\mathbb Q$.

Finally, we need to find a supercompact.  In $V$, let $\kappa < \delta$ be given by Woodin for supercompactness with respect to $A$.  Let $\lambda > \kappa$ be an inaccessible limit point of $A$.  Let $U$ be a normal $\kappa$-complete ultrafilter on $\p_\kappa(\lambda)$ such that $j_U(A) \cap \lambda = A \cap \lambda$.  We have that $j_U(\mathbb P_\kappa) = \mathbb P_\kappa * \mathbb P_\lambda / \mathbb P_\kappa * \mathbb Q$, for some $\mathbb Q$ that is forced to be $\lambda^+$-closed.  Let $G_\delta \subseteq \mathbb P_\delta$ be generic, and let $G_\lambda = G_\delta \restriction \mathbb P_\lambda$.  By $\GCH$, $j_U(\kappa) < \lambda^{++}$ and $j_U(\lambda^{++}) = \lambda^{++}$, so we may build $H \subseteq \mathbb Q$ in $V[G_\lambda]$ that is generic over $M[G_\lambda]$.  Thus we can extend the embedding to $j : V[G_\kappa] \to M[G_\lambda * H]$.  Since $M[G_\lambda * H]$ is $\lambda$-closed in $V[G_\lambda]$ and $\mathbb P_\lambda / G_\kappa$ is $\kappa$-directed-closed, there is $p \in j_U(\mathbb P_\lambda)/(G_\lambda * H)$ below $j[G_\lambda / G_\kappa]$.  Since $|\mathbb P_\lambda| = \lambda$ and  $j_U(\lambda^+) <  \lambda^{++}$, we can build $K \subseteq j_U(\mathbb P_\lambda)/(G_\lambda * H)$ below $p$ in $V[G_\lambda]$ that is generic over $M[G_\lambda * H]$.  Thus we can extend the embedding to $j : V[G_\lambda] \to M[G_\lambda * H * K]$.  This shows that $\kappa$ is $\lambda$-supercompact in $V[G_\lambda]$, a property that is preserved by $\mathbb P_\delta / G_\lambda$.  Thus, $V_\delta[G_\delta] \models$ ``There is a supercompact cardinal.''
\end{proof}

Let us work in a model satisfying the conclusion of the above lemma.  We define by induction on $1 \leq n \leq \omega$ the class of ``order-$n$ Gitik-Sharon forcings'' (abbreivated by $\text{GS}_n$).  Formally, we fix a large enough regular $\theta$ and define these inductively as subsets of $H_\theta$, but it will be clear that choice of $\theta$ is irrelevant, and for $\theta < \theta'$, $\text{GS}_n^{H_\theta} = \text{GS}_n^{H_{\theta'}} \cap H_\theta$. 
Each order-$n$ forcing will add a club of ordertype $\omega^n$ to a large cardinal $\kappa$, consisting of former inaccessibles, while preserving $\kappa$ as a cardinal, collapsing $\kappa^{+\omega \cdot n}$ to $\kappa$, and preserving larger cardinals.

$\text{GS}_1$ is the collection of forcings of the form $\mathbb P(\mu,\vec\omega,\vec{\omega^2},\vec U,\vec K)$, as defined in the previous section, where $\vec \omega$ is the identity sequence $\la 1,2,3,\dots \ra$, and $\vec{\omega^2}$ is the constant sequence $\la \omega,\omega,\omega,\dots\ra$.   

\begin{definition}
A sequence $d=\la U_\alpha,K_\alpha : \alpha < \omega \cdot n \ra$ is a $\mathrm{GS}_n$-sequence if
\begin{enumerate}
\item There is a $\kappa>\omega$ such that each $U_\alpha$ is a $\kappa$-complete ultrafilter. We call $\kappa$ the \emph{critical point} of the sequence $d$.
 \item For $1\leq n < \omega$, $U_n$ is a normal ultrafilter on $\p_\kappa(\kappa^{+n})$ and for $\omega \leq \alpha <\omega \cdot n$ successor, $U_{\alpha}$ is a normal ultrafilter on $\p_\kappa(H_{\kappa^{+\alpha}}).$
 \item For successor $\alpha <\omega \cdot n$, if $j_\alpha \colon V \to M_\alpha$ is the ultrapower embedding from $U_\alpha$, then $K_\alpha$ is $\col(\kappa^{+\alpha +\omega +2},j_\alpha(\kappa))^{M_\alpha}$-generic over $M_\alpha$.
\end{enumerate}
\end{definition}

A partial order $\mathbb P \in \text{GS}_n$ will be determined by the choice of a $\mathrm{GS}_n$-sequence $d$ and a regular cardinal $\mu < \crit(d)$. Suppose $n>1$ and that we have defined $\text{GS}_m$ for $m < n$, and we have a function defined on pairs $(\mu,d) \in H_\theta$ that outputs a partial order $\mathbb P(\mu,d) \in \text{GS}_m$ whenever $d$ is a sequence of length $\omega \cdot m$ as above and $\mu < \crit(d)$ is regular.

Let $d= \la U_\alpha,K_\alpha : \alpha < \omega \cdot n \ra$ be as above and let $\mu < \crit(d)$ be regular.  Conditions in $\mathbb P(\mu,d) \in \text{GS}_n$ take the form: 
 $$p = \la f_{0},e_1,(x_1,a_1),f_1,e_2,(x_2,a_2),f_2,\dots,e_l,(x_l,a_l),f_l,\vec F \ra.$$
 The \emph{stem} of $p$ is the initial segment obtained by removing $\vec F$.  The \emph{length} of $p$ as above is $l$.  We require:
 \begin{enumerate}
  \item For $1\leq i \leq l$: 
  \begin{enumerate}
  \item $|x_i| < \kappa$, $x_i \prec H_{\kappa^{+\omega \cdot (n-1) + i}}$, $\kappa_i := x_i \cap \kappa$ is inaccessible, the transitive collapse of $x_i$ is $H_{\kappa_i^{+\omega \cdot (n-1) + i}}$, and $\la U_\alpha,K_\alpha : \alpha < \omega \cdot (n-1) \ra \in x_i$.
   \item Let $\pi \colon x_i \to H$ be the transitive collapse map. Put $\pi(\la U_\alpha,K_\alpha : \alpha < \omega \cdot(n-1) \ra) := \la u^i_\alpha,k^i_\alpha : \alpha < \omega \cdot (n-1) \ra := d_i$.  We require that $d_i$ is a $\mathrm{GS}_{n-1}$-sequence, $a_i$ is a sequence of functions $\la b^i_\alpha : \alpha < \omega \cdot (n-1) \ra$ such that $\dom(b^i_\alpha) \in u^i_\alpha$ and $[b^i_\alpha]_{u^i_\alpha} \in k^i_\alpha$.
  \end{enumerate}
  \item $f_{0} \in \col(\mu,\kappa)$, and if $l > 0$, then $\la f_{0} \ra \!^\frown e_1 \!^\frown a_1 \in \mathbb P(\mu,d_1)$, where $f_0\!^\frown e_1$ is the stem of the condition. 
  \item For $1 \leq i < l$, $x_i \in x_{i+1}$, and $\la f_{i}\ra \!^\frown e_{i+1} \!^\frown a_{i+1} \in \mathbb P(\kappa_{i}^{+\omega \cdot n +2},d_{i+1})$, where $f_i \!^\frown e_{i+1}$ is the stem.
   \item $f_l \in \col(\kappa_l^{+\omega\cdot n+2},\kappa).$
  \item $\vec F$ is a sequence of functions $\la F_\alpha : \alpha < \omega \cdot n \ra$ such that for each $\alpha$, $\dom F_\alpha \in U_\alpha$ and $[F_\alpha]_{U_\alpha} \in K_\alpha$.
 \end{enumerate}
 Suppose we have two conditions
 $$p = \la f_{0},e_1,(x_1,a_1),f_1,\dots,e_l,(x_l,a_l),f_l,\vec F \ra;$$
 $$q = \la f'_{0},e'_1,(x'_1,a'_1),f'_1,\dots,e'_m,(x'_m,a'_m),f'_m,\vec F' \ra.$$
 We put $q \leq p$ when:
 \begin{enumerate}
  \item $m \geq l$, and for $1\leq i \leq l$, $x_i = x'_i$.
\item For $i \leq l$, $f'_i \supseteq f_i$.
  \item For $1 \leq i \leq l$, $\la f'_{i-1} \ra \!^\frown e'_i \!^\frown a'_i \leq \la f_{i-1}\ra \!^\frown e_i \!^\frown a_i$ in the relevant partial order from $\text{GS}_{n-1}$.
    \item For $l < i \leq m$, $x'_i \in \dom F_{\omega \cdot (n-1) + i}$ and $f'_i \supseteq F_{\omega \cdot (n-1) + i}(x_i)$.
  \item $\vec F \restriction \omega\cdot(n-1) \in x$ if $x = x'_i$ for $l < i \leq m$, or if $x\in\dom F'_{\omega\cdot(n-1)+k}$ for $k>m$.
  \item Put $f_{k} = F_{\omega \cdot (n-1) + k}(x'_{k})$ for $l<k<m$.  If $l < i \leq m$ and $\pi : x_i \to H$ is the transitive collapse map, then $\la f'_{i-1} \ra \!^\frown e'_i \!^\frown a'_i \leq \la f_{i-1}\ra \!^\frown \pi(\vec F \restriction \omega \cdot (n-1))$.
  \item For each $\alpha < \omega \cdot n$, $\dom F'_\alpha \subseteq \dom F_\alpha$, and for each $x \in \dom F'_\alpha$, $F'_\alpha(x) \supseteq F_\alpha(x)$.
  \end{enumerate}
  
  Finally, we may define the order-$\omega$ forcings which generically stack the order-$n$ forcings for finite $n$.  Everything looks quite similar, except now our sequences of functions $\vec F$ have length $\omega^2$, and stems of length $n>0$ look like stems of length-1 conditions from forcings in $\text{GS}_{n+1}$.
  
\begin{remark} 
Unlike the standard supercompact Radin forcing (such as in \cite{Krueger}), the generic Radin point $x_\alpha$ for limit $\alpha$ is strictly larger than $\bigcup_{\beta < \alpha} x_\beta$. This discontinuity plays an important role in the proof of the Prikry Property.
\end{remark}

  We define some notions to describe the conditions in our forcings.  A \emph{type-1 sequence} is a natural number.  For $n > 1$, a \emph{type-$n$ sequence} is a finite sequence of type-$(n-1)$ sequences.    We can define inductively a partial order on these sequences.  For a type-1 sequence, this is just the usual linear order.  If $s = \la t_1,\dots,t_l \ra$ and $s' =   \la t'_1,\dots,t'_m \ra$ are of type-$n$, then we say $s' \geq s$ when $m \geq l$ and $t'_i \geq t_i$ for $1\leq i \leq m$.  It is easy to see by induction that this ordering is upward-directed.
  
  If $p \in \mathbb P \in \mathrm{GS}_1$, then by the \emph{shape of $p$} we mean its length.  If $s = \la t_1,\dots,t_l \ra$ is a type-$n$ sequence, and $$p = \la f_{0}, e_1,(x_1,a_1),f_1,\dots,e_l,(x_l,a_l),f_l,\vec F \ra \in \mathbb P \in \text{GS}_n,$$ then we say, inductively, that the stem of $p$ \emph{has shape $s$} if each $\la f_{i-1} \ra ^\frown e_i \!^\frown a_i$ has shape $t_i$.  If $s = \la t_1,\dots,t_l\ra$ is such that each $t_i$ is a type-$i$ sequence, and $p \in \mathbb P \in \text{GS}_\omega$ takes the same form as above, then we say $p$ has shape $s$ if each $\la f_{i-1} \ra ^\frown e_i \!^\frown a_i$ has shape $t_i$.  Note that if $q \leq p$, then the shape of $q$ is greater or equal to the shape of $p$ in the ordering on sequences.  Since the shape of a condition only depends on its stem, we will also speak of the shapes of stems and their subsequences.
  
  Suppose $\mathbb P \in \text{GS}_n$ for $n\leq\omega$.  For conditions $p,q \in \mathbb P$, we say $p \leq^* q$ if $p \leq q$ and they have the same shape.  If $p \leq q$ and $\stem(p)$ is an initial segment of $\stem(q)$, then we say $p \preceq q$.  We have an operation $p \wedge \vec F$ defined similarly as before, in the discussion preceding Lemma \ref{dc1}. 
  
\begin{lemma}
\label{decisive coloring}
Suppose $\mathbb P(\mu,d) \in \mathrm{GS}_n$, $\mu>\omega$, and $c : \mathbb P \to 3$ is a decisive coloring.
\begin{enumerate}
\item There is a sequence $\vec F$ such that for every condition $p$, every two $r,r' \preceq p \wedge \vec F$ of the same shape have the same color.
\item For every condition $p$, there is $q \leq^* p$ such that every two  $r,r' \leq q$ of the same shape have the same color.
\end{enumerate}
\end{lemma}

\begin{proof}
 The case $n=1$ was proven in Lemma \ref{dc1}.
Assume $n >1$ and the lemma holds for $\mathrm{GS}_m$, $m<n$.  Let $\mathbb P(\mu,d) \in \mathrm{GS}_n$, with $\crit(d) = \kappa$ and $\omega < \mu <\kappa$.  Like before, we prove (1) by showing the following claim by induction:  For all $l<\omega$ and all decisive colorings of the conditions of length $l$, there is $\vec F_l$ such that for all $m\leq l$ and every condition $p$ of length $m$, every two $r,r' \preceq p \wedge \vec F$ of the same shape and of length $l$ have the same color.  This suffices, since we can find $\vec F$ that is a lower bound to the countably many $\vec F_l$.  Suppose $l=0$ and $c$ is such a coloring.  For every $s \in \col(\mu,\kappa)$, choose if possible some $\vec F_s$ such that $c(\la s \ra ^\frown \vec F_s) >0$.  Using the directed closure of the collapses and diagonal intersections, we may select a single sequence $\vec F$ such that $\la s \ra \wedge \vec F \leq \la s \ra ^\frown \vec F_s$ for all $s$.

Suppose the claim is true for $m < l$.  Let $c$ be a decisive coloring of the conditions of length $l$.  For each stem $s = \la f_{0}^s,\dots,(x^s_{l-1},a^s_{l-1}),f^s_{l-1}\ra$ of length $l-1$, and each candidate $(x,a)$ for the last node in a one-step extension containing $s$, we can define a coloring $c_{s,x}$ on conditions of the form $\la f^s_{l-1}\ra \!^\frown e ^\frown a \in \mathbb P(\kappa_{l-1}^{+\omega\cdot n +2},d_x)$ as follows.  First, as in the proof of Lemma \ref{dc1}, we find a sequence $\vec F = \la F_\alpha : \alpha < \omega \cdot n \ra$ such that for each stem $s$, each $x \in \dom F_{\omega \cdot (n-1) + l}$, and each choice of $e$ and $a$ such that there are $f$ and $\vec H$ such that $s ^\frown \la e,(x,a),f \ra ^\frown \vec H$ is a condition below $s ^\frown \vec F$ with color $>0$, then already $s ^\frown \la e,(x,a),F_{\omega \cdot (n-1) + l}(x) \ra ^\frown \vec F$ has this color.  
We can then define $c_{s,x}(\la f^s_{l-1}\ra \!^\frown e ^\frown a) = c(s ^\frown \la e,(x,a), F_{\omega \cdot (n-1) +l}(x) \ra ^\frown \vec F)$.  By the induction hypothesis on the order of Gitik-Sharon forcing, for each such $s,x$, there is a choice of $a_{s,x}$ such that $c_{s,x}(\la f^s_{l-1}\ra \!^\frown e ^\frown a_{s,x})$ depends only on the shape of $e$, for conditions below $\la f^s_{l-1}\ra ^\frown a_{s,x}$.  For each $x$, we can use diagonal intersections to select a sequence $a_x$ such that for all $s$, $\la f^s_{l-1} \ra \wedge a_x \leq \la f^s_{l-1} \ra \!^\frown a_{s,x}$.

  In the ultrapower by $U=U_{\omega \cdot (n-1) + l}$, the function $x \mapsto a_x$ represents a sequence of functions $\vec G$ strengthening $\vec F \restriction \omega \cdot (n-1)= \pi(j_U(\vec F \restriction \omega \cdot (n-1)))$, where $\pi$ is the transitive collapse of $j_U[H_{\kappa^{+\omega\cdot (n-1)+l}}]$.  Let $\vec F'$ be $\vec F$ with the intial segment below $\omega \cdot (n-1)$ replaced by $\vec G$.
  Thus we have for each stem $s$ of length $l-1$, a set $A_s \in U_{\omega \cdot (n-1) + l}$ such that for all $x \in A_s$, $a_x = \pi_x(F' \restriction \omega \cdot (n-1))$, and the color of $s ^\frown \la e,(x,a_x),F'_{\omega \cdot (n-1) +l}(x) \ra ^\frown \vec F'$ depends only on the shape of $e$, if $\la f^s_{l-1} \ra ^\frown e ^\frown a_x \leq \la f^s_{l-1} \ra ^\frown a_x$.  Let $A^*$ be the diagonal intersection of the $A_s$, and let $\vec F''$ be $\vec F'$ restricted to $A^*$ at coordinate $\omega \cdot (n-1) +l$.
  
  Now for any condition $p$ of shape $\la t_1,\dots,t_{l-1}\ra$, the color under $c$ of any $q \preceq p \wedge \vec F''$ of shape $\la t_1,\dots,t_{l-1},t\ra$ depends only on $t$.  So for each type-$(n-1)$ sequence $t$, let $c_t$ color the length-$(l-1)$ conditions accordingly.  Note that each $c_t$ inherits decisiveness from $c$.
   By the induction hypothesis, for each $t$, there is a sequence $\vec F_t$ such that for all $m < l-1$ and all $p$ of length $m$, every $q \preceq p \wedge \vec F_t$ of length $l-1$ has a color under $c_t$ depending only on the shape of $q$.  If $\vec F'''$ is a lower bound to the countably many sequences $\vec F_t$, then $\vec F'''$ satisfies the inductive claim for $l$.  This concludes the argument for (1).
   
  To show (2), let us assume inductively that it holds for $\mathrm{GS}_m$, $m<n$.  Let $\mathbb P \in \mathrm{GS}_n$, let $c : \mathbb P \to 3$ be decisive, and let $\vec F$ be given by (1).  
 Let $p \in \mathbb P$, with $\stem(p) = \la f_{0},\dots,f_{l-1},e_l,(x_l,a_l),f_l\ra$.  
    For every end-extension $q = \stem(p) ^\frown s ^\frown \vec F$ of $p \wedge \vec F$, the color of $q$ depends only on the shape of $s$.  Using the closure of $\col(\kappa_l^{+\omega\cdot n + 2},\kappa)$, we can find $f'_l \supseteq f_l$ such that for every strengthening $s$ of the initial segment of $\stem(p)$ before $f_l$, and every type-$n$ sequence $t$, if there is $f \supseteq f'_l$ such that some $s'$ of shape $t$ with $s ^\frown \la f\ra ^\frown s' \!^\frown \vec F \leq p \wedge \vec F$ has color $>0$, then already $s ^\frown \la f'_l\ra ^\frown s' \!^\frown \vec F$ has this color.
  
 Now for each type-$n$ sequence $t$, and each strengthening $s$ of $\stem(p)$ before $f_{l-1}$, we have a coloring $c_{s,t}$ of the conditions
  $\la f \ra ^\frown e ^\frown a \leq \la f_{l-1} \ra ^\frown e_l ^\frown a_l$ according to the color under $c$ of 
  $s ^\frown \la f,e,(x_l,a),f'_l \ra ^\frown s' \!^\frown \vec F,$
  where $s'$ is anything of shape $t$, such that the resulting condition is below $p \wedge \vec F$.  Using the inductive hypothesis and the weak closure of $\mathbb P(\kappa_{l-1}^{+\omega \cdot n +2},d_l)$, we find $\la f'_{l-1} \ra \!^\frown e'_l \!^\frown a'_l \leq^* \la f_{l-1} \ra ^\frown e_l ^\frown a_l$ such that any two extensions of the former of the same shape have the same color under every $c_{s,t}$.  As a result, we have that for any $s$ strengthening $\stem(p)$ before $f_{l-1}$, for any two $r,r'$ of the same shape below $s ^\frown \la f'_{l-1},e'_l,(x_l,a'_l),f'_l \ra \!^\frown \vec F$,
 for which $s$ is an initial segment of both, $c(r) = c(r')$.  We continue this process in the same fashion down the stem of $p$, in a total of $l$ steps, so that at step $k \leq l$, we find $\la f'_{l-k-1} \ra \!^\frown e'_{l-k} \!^\frown a'_{l-k} \leq^* \la f_{l-k-1} \ra ^\frown e_{l-k} \!^\frown a_{l-k}$, such that for every strengthening $s$ of the initial segment of $\stem(p)$ before $f_{l-k-1}$, any two conditions $r,r'$ of the same shape, with $s$ as an initial segment, and below $s ^\frown \la f'_{l-k-1},e'_{l-k},(x_{l-k},a'_{l-k}),f'_{l-k},\dots,(x_l,a'_l),f'_l,\vec F\ra,$
we have $c(r) = c(r')$.  Eventually we reach the desired condition $q \leq^* p$.
 
 The inductive argument for $\mathrm{GS}_\omega$ is entirely similar.
\end{proof}

\begin{corollary}\label{cor: prikry property GSn}
If $\mathbb P(\mu,d) \in \mathrm{GS}_n$, $1 \leq n \leq \omega < \mu$, then $\la \mathbb P(\mu,d),\leq,\leq^* \ra$ is a Prikry-type forcing.
Furthermore, for a condition $p_0$ of the form 
$$\la f_{0},e_1,(x_1,a_1),f_1,\dots,e_m,(x_m,a_m),f_m,\vec F \ra,$$ 
$\mathbb P(\mu,d) \restriction p_0$ is canonically isomorphic to
$$ \mathbb P(\mu,d_1) \restriction \la f_0 \ra ^\frown e_1 \!^\frown a_1  \times\dots\times \mathbb P(\kappa_{m-1}^{+\omega\cdot n+2},d_m) \restriction \la f_{m-1} \ra ^\frown e_m \!^\frown a_m  \times \mathbb Q,$$
where $\mathbb Q$ is a weakly $\kappa_m^{+\omega\cdot m+2}$-closed Prikry-type forcing.
\end{corollary}

\begin{proof}
Let $\sigma$ be a sentence in the forcing language, and color conditions 0 if they do not decide $\sigma$, 1 if they force $\sigma$, and 2 if they force $\neg\sigma$.  Let $p \in \mathbb P(\mu,d)$, and let $q \leq^* p$ be such that any two extensions of $q$ of the same shape have the same color.  If $q$ does not decide $\sigma$, then by the fact that the ordering on sequences is upward-directed, we can find $r,r' \leq q$ of the same shape that force opposite decisions about $\sigma$, a contradiction.

For the second claim, the map is the obvious one, where the elements of $\mathbb Q$ are the tail-ends beyond place $m$, of conditions below $p_0$.  Let us write $\mathbb P(\mu,d) \restriction p_0$ as $\mathbb R \times \mathbb Q$.  From any decisive coloring $c$ of the conditions in $\mathbb Q$, we can define a decisive coloring $c'$ of $\mathbb R \times \mathbb Q$ by setting $c'(r,q) = c(q)$.  Given any $q \in \mathbb Q$, we can find $p \leq^* (1,q)$ such that any two $p',p'' \leq p$ of the same shape have the same color under $c'$.  This means that any two $q',q'' \leq q$ of the same shape have the same color under $c$.  Then we apply the argument of the previous paragraph.
\end{proof}

If $\mathbb P(\mu,d) \in \mathrm{GS}_n$ for $1 < n < \omega$, with $\crit(d)= \kappa$, and $G \subseteq \mathbb P(\mu,d)$ is generic over $V$, then we have a sequence $\la x_i,G_i : 1\leq i < \omega \ra$ such that:
\begin{enumerate}
\item Each $x_i$ is a typical point in $\p_\kappa(H_{\kappa^{+\omega \cdot (n-1) + i}})$.
\item $\la x_i : 1\leq i < \omega \ra$ is $\in$- and $\subseteq$-increasing, with $\bigcup_i x_i = H_{\kappa^{+\omega \cdot n}}$.
\item $G_1$ is $\mathbb P(\mu,d_1)$-generic, and for $i >1$, $G_i$ is $\mathbb P(\kappa_{i-1}^{+\omega \cdot n +2},d_i)$-generic, where $\kappa_i$ and $d_i$ are as in the definition of $\mathrm{GS}_n$.
\end{enumerate}
From $\la x_i,G_i : 1\leq i < \omega \ra$, we can recover $G$ as the collection of all conditions
$\la f_0,e_1,(x_1,a_1),f_1,\dots,e_l,(x_l,a_l),f_l,\vec F \ra$ such that:
\begin{enumerate}
\item $\la x_1,\dots,x_l\ra$ is an initial segment of $\la x_i : 1 \leq i < \omega \ra$.
\item For $i > l$, $F \restriction \omega\cdot(n-1) \in x_i \in \dom F_{\omega \cdot n + i}$.
\item For $1\leq i \leq l$, $\la f_{i-1} \ra ^\frown e_i ^\frown a_i \in G_i$.
\item Putting $f_i = F_{\omega \cdot n + i}(x_i)$ for $i > l$, $\la f_{i-1} \ra ^\frown \vec F \restriction \omega\cdot(n-1) \in G_i$.
\end{enumerate}
We need the following characterization of genericity, proof of which is essentially the same as for Lemma \ref{genchar}:

\begin{lemma}\label{lemma: genericity for GS_n}
Suppose $d = \la U_\alpha,K_\alpha : \alpha < \omega \cdot n \ra$ and $\mathbb P(\mu,d) \in \mathrm{GS}_n$, with $\omega < \mu < \crit(d) = \kappa$. 
Suppose in some outer model $W \supseteq V$, there is a sequence $\la x_i,G_i : 1\leq i < \omega \ra$ as above.  

Then this sequence generates a $V$-generic filter $G$ for $\mathbb P(\mu,d)$ iff for every sequence $\vec F = \la F_\alpha : \alpha < \omega \cdot n \ra$ such that $\la \emptyset \ra^\frown \vec F$ is a condition, there is $m < \omega$ such that for all $k \geq m$, $\vec F \restriction \omega \cdot(n-1) \in x_k \in \dom F_{\omega \cdot (n-1) + k}$, and 
$$\la F_{\omega \cdot (n-1) + k}(x_{k}) \ra ^\frown \pi_{k+1}( \vec F \restriction \omega \cdot (n-1)) \in G_{k+1},$$
where $\pi_{k+1}$ is the transitive collapse of $x_{k+1}$.
\end{lemma}

To prove the main theorem, we will show by induction that, in a model satisfying the conclusion of Lemma \ref{prepmodel}, if $\mu = \nu^{+\omega\cdot k+2}$ and $\mathbb P(\mu,d) \in \mathrm{GS}_n$, for $1 \leq k,n < \omega$, then $\mathbb P(\mu,d)$ forces that $(\nu^{+\alpha+1},\nu^{+\alpha}) \chang (\nu^{+\beta+1},\nu^{+\beta})$ holds for all limit ordinals $\omega \leq \beta < \alpha \leq \omega^{n+1}$.  Note that we include the case $\nu = 0$ so that the lower pair may be $(\aleph_1,\aleph_0)$.

For the base case, suppose $\mu = \nu^{+\omega\cdot k +2}$, for $1\leq k <\omega$, and $\mathbb P(\mu,d) \in \mathrm{GS}_1$, with $\crit(d) =\kappa$.  By Lemma \ref{gs1 factor} and the preservation claim of Lemma \ref{prepmodel}, we have that in $V^{\mathbb P(\mu,d)}$, $(\nu^{+\omega\cdot i +1},\nu^{+\omega\cdot i}) \chang (\nu^{+\omega\cdot j +1},\nu^{+\omega\cdot j})$ holds for all $1 \leq j < i < \omega$.  Using again the fact that for $\alpha < \kappa$, $(\kappa^{+\omega+1},\kappa^{+\omega}) \chang (\alpha^{+\omega+1},\alpha^{+\omega})$ is indestructible by any $\alpha^{+\omega+2}$-closed forcing of size $\kappa$, the iterated ultrapower construction in the previous section
shows that $\mathbb P(\mu,d)$ also forces $(\kappa^+,\kappa) \chang (\nu^{+\omega\cdot i +1},\nu^{+\omega\cdot i})$ for $1 \leq i < \omega$. 

Assuming that the inductive claim holds for $n$, let us first argue for the weaker claim that if $\mu = \nu^{+\omega \cdot k +2}$, for $1\leq k < \omega$, and $\mathbb P(\mu,d) \in \mathrm{GS}_{n+1}$, then $\mathbb P(\mu,d)$ forces $(\nu^{+\alpha+1},\nu^{+\alpha}) \chang (\nu^{+\beta+1},\nu^{+\beta})$ to hold for all limit ordinals $\omega \leq \beta < \alpha < \omega^{n+2}$ (where the last inequality is strict).  A generic $G \subseteq \mathbb P(\mu,d)$ introduces a Prikry sequence of generics for $\mathrm{GS}_n$ forcings, $\la G_i : 1\leq i < \omega \ra$, where $G_1$ is generic for $\mathbb P(\mu,d_1)$, and for $i \geq 2$, $G_i$ is generic for $\mathbb P(\kappa_{i-1}^{+\omega\cdot(n+1) + 2},d_i)$.  In $V[G_1]$, $\kappa_1 =\nu^{+\omega^{n+1}}$, its successor is $(\kappa_1^{+\omega\cdot n +1})^V$, and we have $(\nu^{+\alpha+1},\nu^{+\alpha}) \chang (\nu^{+\beta+1},\nu^{+\beta})$ for all limit ordinals $\omega \leq \beta < \alpha \leq \omega^{n+1}$.  This is preserved by adjoining $\la G_j : 2 \leq j < \omega \ra$, which adds no subsets of $(\kappa_1^{+\omega\cdot n +1})^V$.
 For $i > 1$, we have that in $V[G_i]$,
$$(\kappa_{i-1}^{+\omega\cdot n+\alpha + 1},\kappa_{i-1}^{+\omega\cdot n +\alpha}) \chang (\kappa_{i-1}^{+\omega\cdot n+\beta + 1},\kappa_{i-1}^{+\omega\cdot n+\beta}),$$
holds for all limit ordinals $0\leq\beta<\alpha\leq\omega^{n+1}$.  For each such $i$, these instances of Chang's Conjecture are preserved by adjoining $\la G_j : i < j < \omega \ra$, which adds no subsets of $(\kappa_{i}^{+\omega\cdot n+1})^V$, the $(\omega^{n+1}+1)^{st}$ cardinal above $\kappa_{i-1}$ in the extension, and also by adjoining $G_1 \times\dots\times G_{i-1}$, which is generic for a $\kappa_{i-1}^{+\omega\cdot n}$-centered forcing.  By the transitivity of Chang's Conjecture, we can combine finitely many instances to bridge the different intervals that lie between adjacent Prikry points, and get the weaker conclusion for $n+1$.

The hard part is to improve the final inequality to allow $\alpha = \omega^{n+2}$.  If the critical point of $d$ as above is $\kappa$, then by applying transitivity again, it suffices to show that the extension satisfies $(\kappa^+,\kappa) \chang (\kappa_i^+,\kappa_i)$ for infinitely many $i$.  Towards this, we generalize Claim \ref{claim: generic over ultrapower, simple} and produce an iterated ultrapower for which we can find a generic filter for (the image of) a forcing $\mathbb{P} \in \mathrm{GS}_{n+1}$.

\begin{claim}\label{claim: generic over ultrapower, general}
Suppose $1\leq n < \omega$, $W$ is a model of $\ZFC$, and $\mathbb{P}(\mu, d) \in \mathrm{GS}_n^W$, with $\crit(d) = \kappa$.  Suppose $p\in \mathbb{P}(\mu,d)$ is a condition of length $l$,
$p = \la f_0,\dots,f_l \ra ^\frown \vec F$.  If $l >0$, let $\nu =  \kappa_l^{+\omega\cdot n +2}$ and let $\mathbb R$ be such that $\mathbb P(\mu,d) \restriction p \cong  \mathbb R \times \mathbb Q$, as in Corollary \ref{cor: prikry property GSn}.  Otherwise let $\nu = \mu$ and let $\mathbb R$ be trivial.

There is an elementary embedding $j\colon W\to W'$, where $W'$ is transitive, $\crit j = \kappa$, $j(\kappa) = \kappa^{+\omega\cdot n}$, and $\kappa^{+\omega\cdot n +1}$ is a fixed point of $j$.
If there is a $W$-generic filter 
$H\subseteq \mathbb R \times \col(\nu,\kappa),$
then there is a $W'$-generic filter $G \subseteq j(\mathbb{P}(\mu,d))$ in $W[H]$ such that $j(p)\in G$.
Moreover, $W[H]$ and $W'[G]$ have the same $\kappa$-sequences of ordinals.
\end{claim}
\begin{proof}
First, let us introduce a temporary notation in order to describe generic filters for $\mathbb{P}(\mu, d)$. 
Every ordinal $\alpha < \omega^\omega$, can be represented using Cantor Normal Form as a sum
\[\alpha = \omega^{m} \cdot k_{m} + \cdots + \omega \cdot k_1 + k_0,\]
where $k_i < \omega$ for all $i$. For $\alpha \neq 0$, let $n_\star(\alpha) = \min \{r \mid k_r \neq 0\}$ and let $m_\star(\alpha) = k_{n_\star(\alpha)}$. 

A generic $G \subseteq \mathbb P(\mu,d)$ can be unraveled into a sequence $\la x_\alpha : 1\leq \alpha < \omega^n \ra \subseteq \p_\kappa(H_{\kappa^{+\omega\cdot n}})$ and filters $\la C_\alpha : \alpha < \omega^n \ra$, from which we can recover $G$.  If $\rho_\alpha = x_\alpha \cap \kappa$, then the $\rho_\alpha$ are increasing, continuous, and cofinal in $\kappa$.  $C_0$ is generic for $\col(\mu,\rho_1)$, and for $\alpha \geq 1$, $C_\alpha$ is generic for $\col(\rho_{\alpha}^{+\omega \cdot n_\star(\alpha) + \omega + 2}, \rho_{\alpha + 1})$.  If $\beta < \alpha$ and $n_\star(\beta) \leq n_\star(\alpha)$, then $x_\beta \in x_\alpha$.

Let us note that by unraveling the criteria for being in the filters associated to the sequences, we can recover $G$ in the following way.  Let  $\vec F = \la F_\alpha : \alpha < \omega \cdot n \ra$ be a sequence of functions.  For each $\alpha < \omega^n$, define a finite sequence 
$\la F_{\alpha}^0,\dots, F_{\alpha}^{n-n_\star(\alpha)-1}\ra$ 
by putting $F_{\alpha}^0 = F_{\omega \cdot n_\star(\alpha)+m_\star(\alpha)}$, and for $0<k< n-n_\star(\alpha)$, $F_\alpha^k = \pi(F_\alpha^{k-1})$, where $\pi$ is the transitive collapse of $x_{\alpha+\omega^{n-k}}$, if that object is in $\dom \pi$.  Put $F'_\alpha= F_\alpha^{n-n_\star(\alpha)-1}$. Then we have $\la \emptyset \ra ^\frown \vec F \in G$ iff for all $\alpha < \omega^n$, $F'_{\alpha}$ is defined, $x_\alpha \in \dom  F'_\alpha$, and $F'_\alpha(x_\alpha) \in C_\alpha.$

Given a $\mathrm{GS}_n$-sequence $d$, let us construct an iterated ultrapower and a sequence $\langle x_\alpha, C_\beta \mid 1 \leq \alpha < \omega^n, \beta < \omega^n\rangle$ as above. We will assume, by induction on $n$ (simultaneously for all models of $\ZFC$, all $\mathrm{GS}_n$-sequences $d$ and all generics $H$) that this process provides a generic filter for the limit ultrapower.

Let $\mu,d,H,W$ be as hypothesized, and let $d = \langle U_\alpha, K_\alpha \mid \alpha < \omega \cdot n\rangle$.
Let us define by induction on $\omega^{n-1} \cdot l <  \alpha \leq \omega^n$, a model $N_\alpha$ and elementary embeddings $j_{\beta,\alpha} \colon N_\beta \to N_\alpha$. The choice of the measures which are applied at each step resembles the iterated ultrapower for obtaining a Radin generic filter (see \cite{MitchellHandbookInnerModels}). 

Let $\alpha_0 = \omega^{n-1}\cdot l +1$, and let $N_{\alpha_0} = W$.
For limit ordinals $\alpha$, let $N_\alpha$ be the direct limit of the system $\langle N_\beta, j_{\beta, \gamma} \mid \beta < \gamma < \alpha\rangle$ and let $j_{\beta,\alpha}$ be the corresponding limit embeddings.
For $\alpha = \beta + 1$, let $j_{\beta, \alpha} \colon N_\beta \to N_\alpha \cong \Ult(N_\beta, j_{\alpha_0,\beta}(U_{\omega \cdot n_\star(\beta) + m_\star(\beta)}))$, and let $j_{\gamma, \alpha} = j_{\beta,\alpha} \circ j_{\gamma, \beta}$ for $\gamma < \beta$. By Fact \ref{fact: gaifman}, $N_{\omega^n}$ is well-founded.  By counting arguments similar to those in the previous section, we can show that $j_{\alpha_0,\omega^n}(\kappa) = \kappa^{+\omega \cdot n}$, and $j_{\alpha_0,\omega^n}(\kappa^{+\omega\cdot n+1}) = \kappa^{+\omega\cdot n+1}$.

Let us define a sequence of filters $\la C_i \mid i < \omega^n \ra$ and a sequence of sets $\la x_i \mid 1 \leq i < \omega^n \ra$.
For $i \leq \omega^{n-1} \cdot l$ we extract $C_i$ and $x_i$ from the $W$-generic filter $H$. 

Let us define the Prikry points for $\alpha > \omega^{n-1} \cdot l$. 
Let $X_\alpha = \kappa^{+m_\star(\alpha)}$ if $n_\star(\alpha) = 0$, and $X_\alpha = H_{\kappa^{+\omega \cdot n_\star(\alpha) + m_\star(\alpha)}}$ otherwise.  Let $y_{\alpha}  =  j_{\alpha, \omega^n} [ j_{\alpha_0,\alpha}\left(X_\alpha \right) ].$
Note that 
$y_{\alpha} = j_{\alpha + 1, \omega^{n}}\left(j_{\alpha,\alpha + 1} [ j_{\alpha_0,\alpha}(X_\alpha) ] \right),$
and in particular it is in $N_{\omega^n}$.
In other words, we take $y_{\alpha}$ to be the seed of the measure $j_{\alpha_0,\alpha}(U_{\omega \cdot n_\star(\alpha)+m_\star(\alpha)})$, 
pushed by the map $j_{\alpha + 1, \omega^{n}}$ to the limit model $N_{\omega^n}$. Since the critical point of the elementary map $j_{\alpha+ 1, \omega^{n}}$ is above the cardinality of $y_\alpha$, it acts pointwise.

If $n_\star(\alpha) = n - 1$, let $x_\alpha = y_\alpha$.  Otherwise, let $\pi$ be the Mostowski collapse of $y_{\alpha+\omega^{n_\star(\alpha)+1}}$ and let $x_\alpha = \pi(y_\alpha)$.  Let $C_\alpha = j_{\alpha_0,\alpha}(K_{\omega \cdot n_\star(\alpha) + m_\star(\alpha)})$.
Let us verify that the obtained filter satisfies the requirements of Lemma \ref{lemma: genericity for GS_n}.

Let $m > l$. Let $G_{m}$ be the filter for the forcing $\mathbb{P}(\rho_{\omega^{n-1} \cdot (m-1)}^{ +\omega\cdot n + 2}, d_{m})^{N_{\omega^{n-1} \cdot m}}$, where $d_{m} = j_{\alpha_0,\omega^{n-1} \cdot m}(d) \restriction \omega \cdot (n-1)$, which is derived from the sequences $\la x_\alpha \mid \omega^{n-1} \cdot (m-1) \leq \alpha < \omega^{n-1} \cdot m \ra$ and $\langle C_\alpha \mid \omega^{n-1} \cdot (m-1) \leq \alpha < \omega^{n-1} \cdot m\rangle$. Let us assume, by induction, that $G_{m}$ is an $N_{\omega^{n-1} \cdot m}$-generic filter. Note that 
\[\mathbb{P}(\rho_{\omega^{n-1} \cdot (m-1)}^{ +\omega\cdot n + 2}, d_m)^{N_{\omega^{n-1} \cdot m}} = \mathbb{P}(\rho_{\omega^{n-1} \cdot (m-1)}^{ +\omega\cdot n + 2}, d_{m})^{N_{\omega^{n}}},\] 
and that $G_{m}$ is also $N_{\omega^n}$-generic. For $m \leq l$, $G_m$ is derived from the $W$-generic filter $H$, and thus it is clearly $N_{\omega^n}$-generic.

Let $z_i = x_{\omega^{n-1} \cdot i}$ for $1\leq i < \omega$.
Let us check that for every sequence $\vec{F}=\langle F_i \mid i < \omega\cdot n\rangle \in N_{\omega^n}$ there is some $k$ such that for all $m > k$, $\vec F \restriction \omega \cdot (n-1) \in z_{m} \in \dom F_{\omega\cdot(n-1)+m}$, and $\langle F_{\omega\cdot(n-1) + m}(z_m)\rangle ^\frown \pi_{m+1}(\vec F \restriction \omega \cdot (n-1)) \in G_{m+1}$. 
Let us show that for $\alpha_0 \leq \alpha < \omega^n$, if $\vec{F} \in N_{\alpha}$ is a sequence of functions such that $\langle\emptyset\rangle^\smallfrown \vec{F}$ is a condition in $j_{\alpha_0,\alpha}(\mathbb P(\mu,d))$, then for every $\beta > \alpha$, 
\begin{align*}&j_{\alpha,\omega^n}(\vec F \restriction \omega\cdot n_\star(\beta)) \in y_\beta \in \dom j_{\alpha,\omega^n}(F_{\omega \cdot n_\star(\beta) + m_\star(\beta)}),\\
 &\text{and } j_{\alpha,\omega^n}(F_{\omega \cdot n_\star(\beta) + m_\star(\beta)})(y_\beta) \in C_\beta.
\end{align*}
The relation $j_{\alpha,\omega^n}(\vec F \restriction \omega\cdot n_\star(\beta)) \in y_\beta$ holds simply because $\vec F \restriction \omega\cdot n_\star(\beta) \in (H_{j_{\alpha_0,\alpha}(\kappa)^{+\omega\cdot n_\star(\beta)+1}})^{N_\alpha}$.
The other claims are true since $\bar y_\beta := j^{-1}_{\beta+1,\omega^n}(y_\beta)$ is the seed of the measure $j_{\alpha_0,\beta}(U_{\omega \cdot n_\star(\beta) + m_\star(\beta)})$ and the domain of $j_{\alpha,\beta}(F_{\omega \cdot n_\star(\beta) + m_\star(\beta)})$ is large with respect to this measure. Moreover, this function represents an element of $j_{\alpha_0,\beta}(K_{\omega \cdot n_\star(\beta) + m_\star(\beta)})$. But 
\begin{align*}
& j_{\alpha, \beta + 1}(F_{\omega \cdot n_\star(\beta) + m_\star(\beta)})(\bar y_\beta) & = & &\\
& j_{\beta,\beta + 1}(j_{\alpha, \beta}(F_{\omega \cdot n_\star(\beta) + m_\star(\beta)}))(\bar y_\beta) & = & &\\
& [j_{\alpha,\beta}(F_{\omega \cdot n_\star(\beta) + m_\star(\beta)})]_{j_{\alpha_0,\beta}(U_{\omega \cdot n_\star(\beta) + m_\star(\beta)})} & \in & \ j_{\alpha_0,\beta}(K_{\omega \cdot n_\star(\beta) + m_\star(\beta)}) = C_\beta. &
\end{align*}

Note that for $\omega^{n-1}\cdot l < \alpha < \omega^n$, the sequence $\la y_\alpha,y_{\alpha+\omega^{n_\star(\alpha)+1}},\dots,y_{\alpha+\omega^{n-1}} \ra$ is both $\in$- and $\subseteq$-increasing.  Thus to compute $x_\alpha$, we get the same result by taking the image of $y_\alpha$ under the transitive collapse $y_{\alpha+\omega^{n_\star(\alpha)+1}}$, as by first collapsing $y_{\alpha+\omega^{n-1}}$, then collapsing the image of $y_{\alpha+\omega^{n-2}}$, etc., until we take the image of $y_\alpha$ under $n-n_\star(\alpha)-1$ successive collapses.  The point is that the latter process parallels exactly the sequence of collapses applied to a sequence of functions $\vec F$ to determine whether $\la\emptyset\ra^\frown\vec F$ is in the filter generated from the sequences $\la x_\alpha,C_\beta : 1 \leq \alpha < \omega^n,\beta<\omega^n\ra$.

Hence, if 
\begin{align*}&j_{\alpha,\omega^n}(\vec F \restriction \omega\cdot n_\star(\beta)) \in y_\beta \in \dom j_{\alpha,\omega^n}(F_{\omega \cdot n_\star(\beta) + m_\star(\beta)}),\\
 &\text{and } j_{\alpha,\omega^n}(F_{\omega \cdot n_\star(\beta) + m_\star(\beta)})(y_\beta) \in C_\beta,
\end{align*}
then $j_{\alpha,\omega^n}(\vec F)'_\beta \in x_\beta \in \dom j_{\alpha,\omega^n}(\vec F)'_\beta$, and $j_{\alpha,\omega^n}(\vec F)'_\beta(x_\beta) \in C_\beta$.
So if $\vec F \in N_{\alpha}$, the genericity criteria holds for $j_{\alpha,\omega^n}(\vec{F})$ for the cofinal segment above $\alpha$.  Since $N_{\omega^n}$ is a direct limit, the generated filter $G$ is generic. 

We would like to claim now that $N_{\omega^n}[G]$ has the same $\kappa$-sequences as $W[H]$. Indeed, since the forcing that introduces $H$ has cardinality $\kappa$, any sequence of ordinals in $W[H]$ has a name of cardinality $\kappa$ and thus can be coded using a sequence of ordinals of length $\kappa$ from $W$. 

Let $\langle \xi_i \mid i < \kappa\rangle$ be a sequence of ordinals in $W$. In $N_{\omega^n}$, for every ordinal there is a representing function $f_i$, and a finite sequence $s_i \subseteq \la y_\alpha : \omega^{n-1}\cdot l < \alpha < \omega^n \ra$, such that $j_{\alpha_0,\omega^n}(f_i)(s_i) = \xi_i$. By our choices of $x_i$ and $C_i$, the sequence $\la y_\alpha : \omega^{n-1}\cdot l < \alpha < \omega^n \ra$ can be computed from the generic filter $G$. 
Since $j_{\alpha_0,\omega^n}(\langle f_i \mid i < \kappa\rangle)$ and $j_{\alpha_0,\omega^n}(\la s_i \mid i < \kappa\rangle)$ are in $N_{\omega^n}$, and since $\langle y_\alpha \mid \omega^{n-1}\cdot l < \alpha< \omega^n\rangle \in N_{\omega^n}[G]$ we conclude that $\langle \xi_i \mid i < \kappa\rangle \in N_{\omega^n}[G]$.  
\end{proof}

Let us return to the proof of the theorem.  Recall that, assuming the inductive claim holds for $\mathrm{GS}_{n}$, we must only show that for every $\mathbb P(\mu,d) \in \mathrm{GS}_{n+1}$ with $\crit(d) = \kappa$, it is forced that $(\kappa^+,\kappa) \chang (\kappa_i^+,\kappa_i)$ holds for infinitely many $i$.  Let $p$ be a condition of length $l$, let $H \subseteq \mathbb R \times \col(\nu,\kappa)$ be as in Claim \ref{claim: generic over ultrapower, general}, with $H$ generic over $V$.  Note that $V[H] \models |(\kappa_l^{+\omega\cdot n})^V| = \kappa_l$, and $(\kappa_l^{+\omega\cdot n+1})^V = \kappa_l^+$.  By Lemma \ref{prepmodel}, $(\kappa^{+\omega\cdot (n+1)+1},\kappa^{+\omega\cdot (n+1)}) \chang (\kappa_l^+,\kappa_l)$ holds in $V[H]$.  Let $j : V \to M$ and $G$ be given by Claim \ref{claim: generic over ultrapower, general}, with $j(p) \in G$.  

Let $\mathfrak A \in M[G]$ be any structure on $j(\kappa^{+\omega\cdot (n+1)+1}) = \kappa^{+\omega\cdot (n+1)+1} = (j(\kappa)^+)^{M[G]}.$  By Chang's Conjecture in $V[H]$, there is a $\mathfrak B \prec \mathfrak A$ of size $\kappa_l^{+\omega\cdot n +1}$ such that $| \mathfrak B \cap \kappa^{+\omega\cdot (n+1)}) | = |\mathfrak B \cap j(\kappa) | = \kappa_l^{+\omega\cdot n}$.  By the closure of $M[G]$, $\mathfrak B \in M[G]$, and thus $M[G] \models (j(\kappa)^+,j(\kappa)) \chang (\kappa_l^+,\kappa_l)$.  By elementarity, the desired conclusion follows.

\section{Chang's Conjecture with the same target}
\label{section:thread}
In this section we will discuss two restricted versions of the Singular Global Chang's Conjecture. 

\begin{theorem}
\label{interval to omega}
Suppose that $\kappa$ is $\nu^+$-supercompact, where $\cf \nu = \kappa^{+}$ and $\nu$ is a limit of measurable cardinals, and $\alpha_\star$ is a countable ordinal.  Then there is a generic extension in which
\[(\mu^+, \mu) \chang (\omega_1, \omega),\]
for all $\mu < \aleph_{\alpha_\star}$.
\end{theorem}
\begin{theorem}
\label{uncountable interval}
Suppose there are two supercompact cardinals and $\alpha_\star>0$ is a countable limit ordinal.  Then there is a generic extension in which 
\[(\mu^+, \mu) \chang (\aleph_{\alpha_\star+1},\aleph_{\alpha_\star}),\]
for all singular $\mu$, $\aleph_{\alpha_\star}<\mu<\aleph_{\omega_1}$.
\end{theorem}

The proof of both theorems follows closely the ideas from \cite{GolshaniHayut2018}, which in turn are motivated by the forcing arguments from \cite{MagidorShelah1994}. 

\begin{proof}[Proof of Theorem \ref{interval to omega}] Let us assume that $\kappa$ is Laver-indestructible (with respect to $\kappa$-directed closed forcing notions of cardinality $\leq\nu^+$) and that $\GCH$ holds above $\kappa$. If this is not the case, we can always force it using Laver forcing \cite{laverindestructible}. Let $\langle \zeta_\beta \mid \beta < \kappa^{+}\rangle$ be a continuous increasing sequence with $\sup \zeta_\beta = \nu$, $\zeta_0 = \kappa$, and $\zeta_{\beta+1}$ measurable for each $\beta < \kappa^+$. 

For every $\alpha < \kappa^{+}$ of countable cofinality, let us pick an increasing cofinal $\omega$-sequence $s_\alpha \colon \omega \to \alpha$. Let us assume that for each $\alpha$, $s_\alpha(0) = 0$, and $s(n)$ is a successor ordinal for $n>0$.

Let us consider the forcing 
\[\mathbb{C}_\alpha = \prod_{n < \omega} \mathbb{E}(\zeta_{s_\alpha(n)},\zeta_{s_\alpha(n + 1)}) \times \col(\zeta_\alpha^+,\nu^+),\]
where $\mathbb{E}(\mu, \delta)$ is the Easton-support product of $\col(\mu, \eta)$ over all inaccessible $\eta < \delta$. The product in the definition of $\mathbb{C}_\alpha$ is taken with full support. For properties of the Easton collapse, see \cite{EastonCollapse}.

For each $\alpha < \kappa^{+}$ of countable cofinality, after forcing with $\mathbb{C}_\alpha$, 
\[\left(\zeta_\alpha^+\right)^V = \left(\kappa^{+\omega + 1}\right)^{V^{\mathbb{C}_\alpha}}.\]
By the arguments of \cite{EskewHayut2018} related to Lemma \ref{lemma: reflecting cc}, there is $\rho_\alpha < \kappa$ such that 
\[V^{\mathbb{C}_\alpha} \models (\kappa^{+\omega + 1}, \kappa^{+\omega}) \chang (\rho_\alpha^+, \rho_\alpha),\]
and this remains true after forcing with $\mathbb D_\alpha = \col(\omega, \rho_\alpha) * \dot{\col}(\rho_\alpha^{+}, {<}\kappa)$. In fact, $(\zeta_\alpha^+,\zeta_\alpha) \chang (\rho_\alpha^+,\rho_\alpha)$ must already hold in $V$, by the distribuitivity of $\mathbb C_\alpha$.

Since the forcing $\mathbb{C}_\alpha$ is weakly homogeneous, the value of $\rho_\alpha$ depends only on $\alpha$ and does not depend on the generic filter for $\mathbb{C}_\alpha$. Therefore, the function $\alpha \to \rho_\alpha$ belongs to the ground model, $V$, and has the property that 
\[1 \Vdash_{\mathbb D_\alpha \times \mathbb C_\alpha} (\kappa^{+\omega + 1}, \kappa^{+\omega}) \chang (\check\rho_\alpha^+, \check\rho_\alpha).\]
By the $\kappa^+$-completeness of $\ns_{\kappa^+}$, there is a stationary set $S\subseteq \kappa^{+}$ and a cardinal $\rho_\star < \kappa$ such that for all $\alpha \in S$, $\rho_\alpha = \rho_\star$.  Let $\mathbb D$ be the common value of $\mathbb D_\alpha$ for $\alpha \in S$.  There is $n_0<\omega$ such that for every club $C \subseteq \kappa^+$, $\{ s_\alpha(n_0) : \alpha \in C \cap S \}$ is unbounded.  By Fodor's Lemma, we may assume that $s_\alpha \restriction n_0$ is constant on $S$.

Let us define a partial order $\mathbb P$ that searches for a ``thread'' of the sequences $s_\alpha$ for $\alpha \in S$.  A condition $t \in \mathbb P$ is a continuous increasing function from a countable successor ordinal $\gamma$ into $\kappa^+$, such that $\ran t \subseteq S \cup \bigcup_{\alpha<\kappa^+} \ran s_\alpha$, and for every limit ordinal $\beta < \gamma$, $\ran s_{t(\beta)} \subseteq \ran t$.  As in \cite{GolshaniHayut2018}, we have:
\begin{claim}
\label{thread}
For every $t \in \mathbb P$, every $\gamma < \omega_1$, and every $\xi < \kappa^+$, there is a stronger condition $t' \supseteq t$ with $\gamma \subseteq \dom t'$ and $s_\beta(n_0) > \xi$ for limit $\beta \in \dom t' \setminus \dom t$.
\end{claim}

In particular, we can find a thread of any countable length.  Let $t$ be a thread of length $\alpha_\star$.  Define a sequence $s : \alpha_\star \to \nu$ as follows.  If $\beta$ is an infinite limit ordinal, then $s(\beta) = \zeta_{t(\beta)}^+$, and otherwise $s(\beta) = \zeta_{t(\beta)}$.  Consider the forcing:

\[\mathbb{C} = \prod_{\beta < \alpha_\star} \mathbb{E}(s(\beta),s(\beta+1)).\]

First let us claim that in the generic extension by 
$\mathbb D \times \mathbb C$, we have $(\aleph_{\beta+1},\aleph_\beta) \chang (\aleph_1,\aleph_0)$ for limit $\beta < \alpha_\star$.  As in \cite{GolshaniHayut2018}, the projection properties of the Levy collapse, together with the fact that $\ran s_\beta \subseteq \ran t$ for limit $\beta < \alpha_\star$, imply that for each limit $\beta < \alpha_\star$, there is a projection $\pi_\beta : \mathbb C_\beta \to \mathbb C$.  If $\mathfrak A$ is a structure on $\zeta_\beta^+$ in $V^{\mathbb D \times \mathbb C}$, then in $V^{\mathbb D \times \mathbb C_\beta}$, there is an elementary $\mathfrak B \prec \mathfrak A$ such that $|\mathfrak B| = \rho_\star^+ = \aleph_1$, and $|\mathfrak B \cap \zeta_\beta | = |\rho_\star | = \aleph_0$.  Since the quotient forcing adds no sets of ordinals of size $<\kappa = \aleph_2$, the instance of Chang's Conjecture holds in $V^{\mathbb D \times \mathbb C}$.

To obtain the result for successors below $\alpha_\star$, we consider instead the forcing $\mathbb D * \dot{\mathbb C}$, where $\dot{\mathbb C}$ is the forcing with the same definition as $\mathbb C$, but constructed in $V^{\mathbb D}$ rather than $V$.  By \cite{Shioya2011}, there is a projection from $\mathbb D \times \mathbb C$ to $\mathbb D * \dot{\mathbb C}$ that is the identity on $\mathbb D$.  By the same argument as above, the relevant instances of Chang's Conjecture at limit ordinals also hold in $V^{\mathbb D * \dot{\mathbb C}}$.

Suppose $\beta < \alpha_\star$ is zero or a successor ordinal.  Let $\zeta=s(\beta) = \zeta_{t(\beta)}$, and let $\eta$ be the predecessor of $\zeta$ in the extension by $\mathbb D * \dot{\mathbb C}$, which is regular.  
Since $\zeta$ is measurable, in the extension 
$$\mathbb D * \prod_{\gamma<\beta} \dot{\mathbb E}(s(\gamma),s(\gamma+1)),$$
there is a normal ideal $I$ on $\zeta$ such that $\p(\zeta)/I$ contains a countably closed dense set---in particular the boolean algebra is a proper forcing.  By \cite{sakaisemiproper}, the following version of Strong Chang's Conjecture holds in this model: If $M$ is a countable elementary submodel of $H_{\zeta^+}$ then there is an elementary $M' \supseteq M$ such that $M \cap \eta = M' \cap \eta$ and $M \cap \zeta \not= M' \cap \zeta$.  By Lemma 15 of \cite{EskewHayut2018}, $(\zeta,\eta) \chang (\aleph_1,\aleph_0)$ is preserved by the formerly $\zeta$-closed quotient $\prod_{\beta\leq\gamma<\alpha_\star}\mathbb E(s(\gamma),s(\gamma+1))$.\end{proof}
\begin{remark}
Note that the assumption that $\nu$ is a limit of measurable cardinals is used in order to get Chang's Conjecture between successors of regulars and $\omega_1$. If we only want Chang's Conjecture to hold between successors of singulars and $\omega_1$, we can drop this assumption. 
\end{remark}

\begin{proof}[Proof of Theorem \ref{uncountable interval}]  Let $\kappa_0 < \kappa$ be supercompact, and let $\alpha_\star>0$ be a fixed countable limit ordinal.  First force Martin's Maximum (MM) while turning $\kappa_0$ into $\aleph_2$, as in \cite{fms1}.  By \cite{larson}, MM is indestructible under $\aleph_2$-directed-closed forcing.  Then, force with Laver's forcing, which is $\aleph_2$-directed-closed, to force that $\kappa$ is indestructibly supercompact and $\GCH$ holds above $\kappa$.

Next we need, for large enough $\mu<\kappa$, a forcing $\mathbb D_\mu$ that turns $\kappa$ into $\aleph_{\alpha_\star+3}$ 
while preserving $\omega_1$ and satisfying the hypotheses of Lemma \ref{lemma: reflecting cc}. If $\tau(\alpha_\star) = \omega$, let $\mathbb D_\mu = \col(\aleph_{\iota(\alpha_\star)+1},\mu) \times \col(\mu^{+\omega+2},{<}\kappa)$.  If $\tau(\alpha_\star)>\omega$, let $\vec \gamma$ be the identity sequence converging to $\omega$, and let $\vec\delta$ be a non-decreasing sequence summing to $\tau(\alpha_\star)$, with $\delta_1 \geq \omega$.  Let $\mathbb D_\mu = \mathbb P(\aleph_{\iota(\alpha_\star)+1},\vec\gamma,\vec\delta,\vec U,\vec K) \times \col(\mu^{+\omega+2},{<}\kappa)$, where $\vec U$ and $\vec K$ are $\omega$-sequences such that $U_n$ is a normal $\mu$-complete ultrafilter on $\p_\mu(\mu^{+n})$, and $K_n$ is 
sufficiently generic filter, as in Section \ref{various}.

Working in a model of MM, let us repeat the arguments from the beginning of the proof of Theorem \ref{interval to omega}. For each $\alpha < \kappa^+$ of countable cofinality, choose a cofinal increasing sequence $s_\alpha : \omega \to \alpha$ with $s_\alpha(0) = \kappa$ and $s_\alpha(n)$ is a double successor ordinal for $n>0$.  For each $\alpha < \kappa^+$ of countable cofinality, define
$$\mathbb C_\alpha = \prod_{n<\omega} \col(\kappa^{+s_\alpha(n)},\kappa^{+s_\alpha(n+1)-1}) \times \col(\kappa^{+\alpha+2},\kappa^{+\kappa^++1}).$$
For each $\alpha$, there is $\mu_{\alpha} < \kappa$ such that
$$1 \Vdash_{\mathbb D_{\mu_\alpha} \times \mathbb C_\alpha}
(\kappa^{+\omega+1},\kappa^{+\omega}) \chang (\aleph_{\alpha_\star+1},\aleph_{\alpha_\star}).$$

As above, let $S \subseteq \kappa^+$ be a stationary set of countable cofinality ordinals such that $\mu_\alpha$ has the same value for all $\alpha \in S$, and that the threading forcing $\mathbb P$ satisfies Claim \ref{thread}.  In particular, there is $n_0<\omega$ such that for all club $C \subseteq \kappa^+$, $\{ s_\alpha(n_0) : \alpha \in S \cap C \}$ is unbounded, and $s_\alpha \restriction n_0$ is the same for all $\alpha \in S$.
Let $\mathbb D = \mathbb D_{\mu_\alpha}$ for any $\alpha \in S$.  We now claim that $\mathbb P$ preserves stationary subsets of $\omega_1$. This is a reminiscent of the forcing for Friedman's Problem (see \cite[Theorem 9]{fms1}). 

Fix a stationary set $A \subseteq \omega_1$ and a condition $t_0 \in \mathbb P$. Let $\dot C$ be a $\mathbb P$-name for a club subset of $\omega_1$, and let 
\[M \prec (H_{\kappa^{++}},\in, \la s_\alpha : \alpha < \kappa^+ \ra, S, \mathbb P,t_0 ,A,\dot C,\lhd)\]
be such that $M \cap \kappa^+ = \delta \in S$, where $\lhd$ is a well-order of $H_{\kappa^{++}}$. Let us assume further that $M$ is the union of an increasing sequence of models $M_n$ such that $M_n \in M_{n+1}$.  We may also assume that $s_\delta(n_0) > \sup(M_0 \cap \delta)$.

Let $N'_n \prec M_n$ be the Skolem hull of the finite set $\ran s_\delta \cap M_n$.  For $\alpha<\omega_1$ and $n<\omega$, let $N'_n[\alpha]$ be the Skolem hull of $N'_n \cup \alpha$.  There is some $\alpha<\omega_1$ such that for all $n<\omega$, $N'_n[\alpha] \cap \omega_1 = \alpha \in A$.  Let $N_n =N'_n[\alpha]$ for such an $\alpha$.  Let $N = \bigcup N_n$, so $N \prec M$ is countable, $\sup(N \cap \kappa^+) = \delta$, $\ran s_\delta \subseteq N$, and $N \cap \omega_1 \in A$.  

Let $\la D_n : n < \omega \ra$ enumerate the dense subsets of $\mathbb P$ in $N$, such that $D_n \in N_n$. Using Claim \ref{thread}, we can build a sequence $t_0 \geq t_1 \geq t_2 \geq \dots$ such that for $n>0$, $t_n \in D_n \cap N_n$ and $\ran s_\delta \cap N_n \subseteq \ran t_n$. We achieve that by working inside $N_n$. We first extend $t_{n-1}$ by the finite set $\ran s_\delta \cap N_n$ and then extend this condition to meet $D_n$.  Let $\gamma = \ot(\bigcup_n t_n)$, and let $t = \bigcup_n t_n \cup \{\la \gamma,\delta \ra\}$. Then $t$ is an $(N,\mathbb P)$-master condition, and so it forces $A \cap \dot C \not= \emptyset$.

Applying MM, we find a thread $t$ of length $\omega_1$.  Let $s: \omega_1 \to \kappa^+$ be such that $s(\alpha) = t(\alpha)+2$ for limit $\alpha>0$ and $s(\alpha) = t(\alpha)$ otherwise.  Let us consider the forcing
$$ \mathbb C = \prod_{\alpha < \omega_1} \col(\kappa^{+s(\alpha)},\kappa^{+s(\alpha+1)-1}).$$
For every $\beta \in S$, there is a projection from $\mathbb C_\beta$ to $\mathbb C$.  Therefore, since the quotient adds no sets of ordinals of size $<\kappa$, $\mathbb D \times \mathbb C$ forces the desired conclusion.
\end{proof}
\begin{remark}
By slightly modifying the proof of Theorem \ref{uncountable interval}, one can strengthen the conclusion of the theorem as follows.  Suppose MM holds and there is a supercompact cardinal.  For every $\beta<\omega_2$ and every nonzero $\alpha_\star<\beta$ of countable cofinality, there is an $\omega_1$-preserving generic extension in which $(\mu^+, \mu) \chang (\aleph_{\alpha_\star + 1}, \aleph_{\alpha_\star})$ for all $\mu < \aleph_{\beta}$, such that $\cf \mu = \omega$ and $\mu > \aleph_{\alpha_\star}$. 
\end{remark}
\section{Open Problems}
The construction in Section \ref{section: singular gcc} is limited to instances of Chang's Conjecture between successors of singular cardinals below $\aleph_{\omega^\omega}$. In order to push this mechanism forwards, one needs to start with a model in which there is a cardinal $\kappa$ which is $\kappa^{+\alpha + 1}$-supercompact and Chang's Conjecture holds between any pair of singular cardinals in the interval $[\kappa, \kappa^{+\alpha}]$. Since our method to produce an interval with such properties with limits of limit cardinals includes Prikry forcing, it cannot preserve supercompactness.
\begin{question}
Is it consistent relative to large cardinals that $(\mu^+,\mu)\chang(\nu^+,\nu)$ holds whenever $\mu$ and $\nu$ have countable cofinality?
\end{question}
The known limitations on Global Chang's Conjecture do not seem to rule out the consistency of a strengthening of Theorem \ref{interval to omega} to a global statement:
\begin{question}
Is it consistent relative to large cardinals that $(\kappa^+,\kappa)\chang(\omega_1,\omega)$ holds for all infinite cardinals $\kappa$?
\end{question}

\bibliographystyle{amsplain}
\bibliography{ref}

\providecommand{\bysame}{\leavevmode\hbox to3em{\hrulefill}\thinspace}
\providecommand{\MR}{\relax\ifhmode\unskip\space\fi MR }
\providecommand{\MRhref}[2]{%
  \href{http://www.ams.org/mathscinet-getitem?mr=#1}{#2}
}
\providecommand{\href}[2]{#2}
\begin{thebibliography}{10}

\bibitem{AbrahamMagidorHandbook}
Uri Abraham and Menachem Magidor, \emph{Cardinal arithmetic}, Handbook of set
  theory. {V}ols. 1, 2, 3, Springer, Dordrecht, 2010, pp.~1149--1227.
  \MR{2768693}

\bibitem{BenNeriaLambieHansonUnger}
Omer Ben-Neria, Chris Lambie-Hanson, and Spencer Unger, \emph{Diagonal
  supercompact {R}adin forcing}, Ann. Pure Appl. Logic \textbf{171} (2020),
  no.~10, 102828, 17. \MR{4103778}

\bibitem{Bukovsky1977}
Lev Bukovsk\'y, \emph{Iterated ultrapower and {P}rikry's forcing}, Comment.
  Math. Univ. Carolinae \textbf{18} (1977), no.~1, 77--85. \MR{0446978}

\bibitem{CummingsForemanMagidor2001}
James Cummings, Matthew Foreman, and Menachem Magidor, \emph{Squares, scales
  and stationary reflection}, J. Math. Log. \textbf{1} (2001), no.~1, 35--98.
  \MR{1838355 (2003a:03068)}

\bibitem{Dehornoy1978}
Patrick Dehornoy, \emph{Iterated ultrapowers and {P}rikry forcing}, Ann. Math.
  Logic \textbf{15} (1978), no.~2, 109--160. \MR{514228}

\bibitem{EskewHayut2018}
Monroe Eskew and Yair Hayut, \emph{On the consistency of local and global
  versions of chang's conjecture}, Transactions of the American Mathematical
  Society \textbf{370} (2018), no.~4, 2879--2905.

\bibitem{fms1}
M.~Foreman, M.~Magidor, and S.~Shelah, \emph{Martin's maximum, saturated
  ideals, and nonregular ultrafilters. {I}}, Ann. of Math. (2) \textbf{127}
  (1988), no.~1, 1--47. \MR{924672}

\bibitem{Foreman2009}
Matthew Foreman, \emph{Smoke and mirrors: combinatorial properties of small
  cardinals equiconsistent with huge cardinals}, Adv. Math. \textbf{222}
  (2009), no.~2, 565--595. \MR{2538021}

\bibitem{ForemanMagidor1995}
Matthew Foreman and Menachem Magidor, \emph{Large cardinals and definable
  counterexamples to the continuum hypothesis}, Annals of Pure and Applied
  Logic \textbf{76} (1995), no.~1, 47--97.

\bibitem{ForemanMagidor97}
Matthew Foreman and Menachem Magidor, \emph{A very weak square principle}, J.
  Symb. Log. \textbf{62} (1997), no.~1, 175--196.

\bibitem{GitikHandbook}
Moti Gitik, \emph{Prikry-type forcings}, Handbook of set theory. {V}ols. 1, 2,
  3, Springer, Dordrecht, 2010, pp.~1351--1447. \MR{2768695}

\bibitem{GitikSharon}
Moti Gitik and Assaf Sharon, \emph{On {SCH} and the approachability property},
  Proc. Amer. Math. Soc. \textbf{136} (2008), no.~1, 311--320. \MR{2350418}

\bibitem{GolshaniHayut2018}
Mohammad Golshani and Yair Hayut, \emph{The tree property on a countable
  segment of successors of singular cardinals}, Fund. Math. \textbf{240}
  (2018), no.~2, 199--204. \MR{3720924}

\bibitem{Krueger}
John Krueger, \emph{Radin forcing and its iterations}, Arch. Math. Logic
  \textbf{46} (2007), no.~3-4, 223--252. \MR{2306177}

\bibitem{larson}
Paul Larson, \emph{Separating stationary reflection principles}, J. Symbolic
  Logic \textbf{65} (2000), no.~1, 247--258. \MR{1782117}

\bibitem{laverindestructible}
Richard Laver, \emph{Making the supercompactness of {$\kappa $} indestructible
  under {$\kappa $}-directed closed forcing}, Israel J. Math. \textbf{29}
  (1978), no.~4, 385--388. \MR{0472529}

\bibitem{MagidorShelah1994}
Menachem Magidor and Saharon Shelah, \emph{When does almost free imply free?
  ({F}or groups, transversals, etc.)}, J. Amer. Math. Soc. \textbf{7} (1994),
  no.~4, 769--830. \MR{1249391}

\bibitem{MitchellHandbookInnerModels}
William~J. Mitchell, \emph{Beginning inner model theory}, Handbook of set
  theory. {V}ols. 1, 2, 3, Springer, Dordrecht, 2010, pp.~1449--1495.
  \MR{2768696}

\bibitem{Perlmutter2015}
Norman~Lewis Perlmutter, \emph{The large cardinals between supercompact and
  almost-huge}, Arch. Math. Logic \textbf{54} (2015), no.~3-4, 257--289.
  \MR{3324126}

\bibitem{sakaisemiproper}
Hiroshi Sakai, \emph{Semiproper ideals}, Fund. Math. \textbf{186} (2005),
  no.~3, 251--267. \MR{2191239}

\bibitem{Shelah91}
Saharon Shelah, \emph{Reflecting stationary sets and successors of singular
  cardinals}, Arch. Math. Logic \textbf{31} (1991), no.~1, 25--53.

\bibitem{ShelahProper}
\bysame, \emph{Proper and improper forcing}, second ed., Perspectives in
  Mathematical Logic, Springer-Verlag, Berlin, 1998. \MR{1623206}

\bibitem{Shioya2011}
Masahiro Shioya, \emph{The {E}aston collapse and a saturated filter}, RIMS
  Kokyuroku \textbf{1754} (2011), 108--114.

\bibitem{EastonCollapse}
\bysame, \emph{Easton collapses and a strongly saturated filter}, Arch. Math.
  Logic \textbf{59} (2020), no.~7-8, 1027--1036. \MR{4159767}

\bibitem{Steel2016}
John Steel, \emph{An introduction to iterated ultrapowers}, Forcing, iterated
  ultrapowers, and Turing degrees (Chitat Chong et~al., eds.), Lect. Notes Ser.
  Inst. Math. Sci. Natl. Univ. Singap., vol.~29, World Sci. Publ., Hackensack,
  NJ, 2016, pp.~123--174.

\end{thebibliography}
\end{document}